\documentclass[letterpaper, 10 pt, conference]{ieeeconf}  

\IEEEoverridecommandlockouts                              

\usepackage{graphicx} 
\usepackage{amsmath,amssymb, amsfonts}
\usepackage{mathptmx} 
\usepackage{times} 
\usepackage{cases}
\usepackage{cite}
\usepackage[hidelinks]{hyperref}
\usepackage{mathtools}
\usepackage{tikz}
    \usetikzlibrary{shapes,arrows,calc,positioning,fit,automata}
\usepackage{xcolor}
\usepackage{xargs}

\title{\LARGE\bf Adaptive Extremum Seeking Control via the RMSprop Optimizer}
\author{Patrick McNamee and Zahra Nili Ahmadabadi 
\thanks{P. McNamee and Z. Nili Ahmadabadi are with the  Department of Mechanical Engineering,
        San Diego State University, 
        San Diego, CA, USA
        {\tt\small pmcnamee5123@sdsu.edu} and 
        {\tt\small zniliahmadabadi@sdsu.edu}}%
\thanks{Thus work was funded by Department of the Navy, Office of Naval Research under ONR award number N000142412269. Any opinions, findings, and conclusions or recommendations expressed in this material are those of the author(s) and do not necessarily reflect the views of the Office of Naval Research.}%
}

\newcommand{\annulus}{\set{A}}

\newcommand{\average}[1]{\overline{#1}}
\newcommand{\avgest}[1]{\estimate{\average{#1}}}
\newcommand{\avgerror}[1]{\tilde{\average{#1}}}

\newcommand{\classK}{\mathcal{K}}
	\newcommand{\classKinfty}{\classK_{\infty}}
	\newcommand{\classKL}{\mathcal{KL}}
\newcommand{\comment}[1]{}
\newcommand{\continuous}{\set{C}}
\newcommand{\costfunction}{{J}}

\newcommand{\critical}[1]{{#1}^{\star}}

\newcommand{\dini}{D}
	\newcommand{\upperrightdini}{\dini^+}

	\newcommand{\lowerleftdini}{\dini_-}
\newcommand{\dithersignal}{p}
    \newcommand{\ditheramplitude}{a}
    \newcommand{\ditherrate}{\omega}

    \newcommand{\relativeamplitude}{r}
    \newcommand{\relativerate}{\ditherrate'}
    
\newcommand{\domain}{\set{D}}

\newcommand{\estimate}[1]{\hat{#1}}

\newcommand{\gain}{{k}}
\newcommand{\gradient}{{g}}
	\newcommand{\altgradient}{\mathfrak{\gradient}}
    \newcommand{\gradientmagnitudesquared}{v}
		\newcommand{\altgradientmagnitudesquared}{\nu}

\newcommand{\gradientfilter}{m}
	
\newcommand{\heavisidestep}{H}
\newcommand{\hessian}{{H}}
\newcommand{\hessianfilter}{N}
\newcommand{\identitymatrix}{I}

\newcommand{\invhessian}{{\Gamma}}

\newcommand{\limitvalue}[1]{#1^{*}}
\newcommand{\lyapunov}{V}
    
\newcommand{\optimal}[1]{{#1}^{*}}

\newcommand{\parameter}{\theta}
    \newcommand{\altparameter}{\vartheta}

\newcommand{\realnumbers}{\mathbb{R}}
	\newcommand{\positivereals}{\realnumbers_{\geq 0}}
	\newcommand{\strictlypositivereals}{\realnumbers_{> 0}}

\newcommand{\rmd}{{d}}
\newcommand{\sensoroutput}{y}
\newcommand{\set}[1]{\mathcal{#1}}

\newcommand{\smoothed}[1]{\tilde{#1}}

\newcommand{\state}{x}
    \newcommand{\altstate}{z}
\newcommand{\stationary}[2][]{#2_{*}}
\DeclareMathOperator{\subto}{s.t.}

\newcommand{\ball}[2][]{\set{B}^{{#1}}_{{#2}}}
\newcommandx{\derivative}[3][1={}, 2={}]{\frac{{\rmd^{#2} #1}}{{\rmd #3^{#2}}}}

\newtheorem{thm}{Theorem}

\newtheorem{asmp}{Assumption}
\newtheorem{defn}{Definition}
\newtheorem{rem}{Remark}

\tikzset{%
  block/.style    = {draw, thick, rectangle, minimum height = 3em,
    minimum width = 3em},
  sum/.style      = {draw, circle}, 
  input/.style    = {coordinate}, 
  output/.style   = {coordinate}, 
  gain/.style = {
  	draw, 
    isosceles triangle,
    isosceles triangle apex angle=50,
    minimum height = 3.0em,
    outer sep=0},
}

\pgfdeclarelayer{background}
\pgfdeclarelayer{foreground}
\pgfsetlayers{background,main,foreground}

\makeatletter
\def\namedlabel#1#2{\begingroup
    #2%
    \def\@currentlabel{#2}%
    \phantomsection\label{#1}\endgroup
}

\begin{document}

\maketitle

\begin{abstract}
	Extremum Seeking Control (ESC) is a well-known set of continuous time algorithms for model-free optimization of a cost function. One issue for ESCs is the convergence rates of parameters to extrema of unknown cost functions. The local convergence rate depends on the second, or sometimes higher, order derivatives of the unknown cost function. To mitigate this dependency, we propose the use of the RMSprop optimizer for ESCs as RMSprop is an adaptive gradient-based optimizer which attempts to have a normalized convergence rate in all parameters. Practical stability results are given for this RMSprop ESC (RMSpESC). In particular notability, the proof of practical stability uses Lyapunov function based on observed contracting, attractive sets. Versions of this Lyapunov function could be applied to other areas of applications, in particular for interconnected systems. 
\end{abstract}

\section{Introduction}
	\label{sec:introduction}

Extremum Seeking Control (ESC) is a family of model-free optimization algorithms in continuous time, whose state dynamics are defined by differential equations. These algorithms have been used in a variety of application areas such as variable cam timing engine operation \cite{ref:popovic-2006}, ABS braking \cite{ref:dincmen-2012}, and speed control on sailing yachts \cite{ref:xiao-2012}. ESCs are fundamentally based on discrete optimization algorithms, where parameters are updated to minimize a cost function. The most basic ESC is the Gradient-based Extremum Seeking Control (GESC) which attempts to mimic a gradient descent algorithm. This baseline approach has local convergence rates towards extremums that are dependent on the eigenvalues of an unknown Hessian of the cost function at the extrema. These eigenvalues could be arbitrarily small or large which may lead for the GESC to be impractical to implement in real world scenarios.
 
Another ESC is the Newton-based ESC (NESC) \cite{ref:ghaffari-2012} which is based on Newton algorithms and attempts to correct for the unknown Hessian by estimating its inverse. This can lead to, in specific circumstances, assignable convergence rates. However, the algorithms require that the Hessian be invertible for convergence. This additional requirement and the need to estimate more derivatives of the cost function tend to complicate their implementation into real-world scenarios. Instead of the Newton algorithms, one can use some existing adaptive algorithms such as the RMSprop \cite{ref:ma-2022}. Adaptive algorithms seek to alter the standard gradient descent methods for some purpose, where the RMSprop method seeks to have a uniform convergence speed in all parameter directions by scaling each gradient component by an estimate of the inverse gradient component magnitude. This uniform convergence speed is an attractive property to try and achieve with ESCs, thus we propose a novel RMSprop ESC (RMSpESC).
	
RMSpESC is comprised into three parts: 1) updating the parameter estimate; 2) estimating the gradient; and 3) estimating the the squared gradient components. While this methodology looks similar to those presented in frameworks for ESCs, such as \cite{ref:nesic-2012}, it can be shown that the practical stability properties of the RMSpESC are qualitatively better than what is given by \cite{ref:nesic-2012}. Namely, the practical stability for a general cost function can be given with respect to at most two small parameters rather than the three small parameters given by \cite{ref:nesic-2012}.

\textit{Main Contribution}: 
	We present RMSpESC and prove its practical stability. We show that RMSpESC is practically stable with respect to two small parameters for differentiable multivariable cost functions and one small parameter for strictly convex scalar cost functions. Both of thse theoretical results are qualitatively better practical stability results than what would have been demostrated by previous literature.

\textit{Notation}:
	We use $\realnumbers$ indicates the real numbers while $\positivereals$ ($\strictlypositivereals$) are strictly non-negative (positive) real numbers. A vector $x\in\realnumbers^n$ is a vector with $n$ real numbers. The set $\continuous^n$ is the set of continuous functions which are at least $n$ times differentiable. The ball of the origin with radius $r$ is denoted as $\ball{r}$ whereas a ball of the same radius but with a center $c$ that is not the origin is denoted $\ball[c]{r}$. The norm $\Vert\cdot\Vert$ is the standard Euclidian norm while $\Vert x \Vert_{\set{Y}} = \inf_{y\in\set{Y}}\Vert x - y \Vert$ is norm of a point from a set.

\section{Problem Statement}
	\label{sec:problem-statement}

The general problem statement of extremum seeking control acting directly on maps is to practically minimize a sensor output $\sensoroutput$ using model-free methods. The sensor output is generally considered as $\sensoroutput = \costfunction(\parameter)$ where the cost function $\costfunction:\realnumbers^n\to\realnumbers$ depends on the parameter vector $\parameter\in\realnumbers^n$ and is minimized by some optimal parameter $\optimal{\parameter}$. This work will use additive dither signals, to decompose the parameter into $\parameter~=~\estimate{\parameter}~+~\ditheramplitude\dithersignal(\ditherrate t)$ where $\estimate{\parameter}$ is the estimate of $\optimal{\parameter}$, $\dithersignal$ is the dither signal, and $\ditheramplitude > 0$ is a small parameter defining the magnitude of the perturbation. Additionally, as is the case with previous literature, this work includes $\ditherrate$ ($\ditherrate^{-1}$) is a large (small) parameter to control the rate of the dither signal. We will assume the following about the dither signal.
\begin{asmp}
	\label{asmp:dither-signal-linearly-independent}
	The dither signal $\dithersignal:\realnumbers\to\realnumbers^n$ is an almost periodic signal that satisfies
	\begin{equation}
		\sup_{\tau\in\realnumbers} \Vert \dithersignal(\tau)\Vert_2 = 1
	\end{equation}
	and the vector $[1,\ \dithersignal(\tau)^\top]^\top$ is comprised of linearly independent signals, that is to say that there does not exist constants $\beta_0,\beta_1,\ldots,\beta_n\in\realnumbers$ such that
	\begin{equation}
		\beta_0 + \sum_{i=1}^n \beta_i \dithersignal_i(\tau) = 0\quad \forall\ \tau\in\realnumbers
	\end{equation}
	other than the trivial solution $\beta_0 = \beta_1=\ldots=\beta_n = 0$.
\end{asmp}
This assumption is minor, it consists of a normalization of $\dithersignal$ and ensures that we can find ``demodulation'' signals to form perturbation-based derivative estimates of the cost function using \cite{ref:james-kavanaugh-2026}. For traditional sinusoidal dither signals commonplace in the previous literature, the dither signal is defined as a vector whose elements are $\dithersignal(\tau) = \relativeamplitude_i \sin(\relativerate_i \tau)$ where $\relativerate_i\in\strictlypositivereals$ controls the relative dither rates and $\relativeamplitude_i$'s controls the relative dither amplitudes. To satisfy the linear independence of $\dithersignal$ in Assumption~\ref{asmp:dither-signal-linearly-independent}, it is necessary to select the relative dither rates such that $\vert \relativerate_i \vert \neq \vert \relativerate_j \vert$ for all $i \neq j$.

While we have given the requirements on $\dithersignal$, the following assumptions are needed about the cost function in this work in order to prove practical stability.
\begin{asmp}
	\label{asmp:continuously-differentiable}
	The cost function $\costfunction:\realnumbers^n\to\realnumbers$ is continuously differentiable ($\costfunction\in\continuous^1$) on the domain $\parameter\in\realnumbers^n$.
\end{asmp}
\begin{asmp}
	\label{asmp:unique-minimum}
	The cost function $\costfunction$ has a unique minimum $\optimal{\parameter}$ such that $\forall\ \parameter\in\realnumbers^n$, $\costfunction(\optimal{\parameter}) < \costfunction(\parameter)$ if $\parameter\neq\optimal{\parameter}$. Without loss of generality, assume that $\optimal{\parameter} = 0$ and $\costfunction(\optimal{\parameter}) = 0$.
\end{asmp}
\begin{asmp}
	\label{asmp:minimum-is-unique-stationary-point}
	The gradient vector $\nabla\costfunction(\parameter) = 0$ if and only if $\parameter = \optimal{\parameter}$.
\end{asmp}
\begin{asmp}
	\label{asmp:cost-function-radially-unbounded}
	The cost function $\costfunction$ is a radially unbounded function.
\end{asmp}
The above assumptions are needed to ensure that the continuous-time optimization problem is well-posed. Additionally, given Assumptions~\ref{asmp:continuously-differentiable}-\ref{asmp:cost-function-radially-unbounded}, we can give the following bound on $\costfunction$:
\begin{rem}
	\label{rem:cost-function-bounded}
	Given Assumptions \ref{asmp:continuously-differentiable}-\ref{asmp:cost-function-radially-unbounded}, there exists $\alpha_{\parameter,1},\alpha_{\parameter,2}\in\classKinfty$ such that
	\begin{equation}
		\alpha_{\parameter,1}(\Vert \parameter\Vert) \leq \costfunction(\parameter) \leq \alpha_{\parameter,2}(\Vert \parameter \Vert)
	\end{equation}
\end{rem}

\section{RMSprop Extremum Seeking Control}
	\label{sec:rmsprop-esc}

The continuous time differential equations defining RMSprop, from \cite{ref:ma-2022}, are
\begin{subequations}
	\label{eq:rmsprop:ct-analogue}
\begin{align}
	\label{eq:rmsprop:parameter-ode}
	\derivative{t} \parameter_i &= \frac{-\gain \frac{\partial\costfunction}{\partial\parameter_i}(\parameter)}{\sqrt{\vert \gradientmagnitudesquared_i \vert  + \varepsilon}} \\
	\label{eq:rmsprop:squared-gradient-estimate-ode}
	\derivative{t} \gradientmagnitudesquared_i &= \omega_{\ell}\left[\left(\frac{\partial\costfunction}{\partial\parameter}(\parameter)\right)^2 - \gradientmagnitudesquared_i \right]
\end{align}
\end{subequations}
for $i=1,2,\ldots,n$. The main idea behind these differential equations is to normalize the descent rate in each parameter $\parameter_i$ by an estimate of the magnitude of the  partial derivative $\frac{\partial\costfunction}{\partial\parameter_i}$. This is accomplished by low pass filtering the squares of the partial derivatives in \eqref{eq:rmsprop:squared-gradient-estimate-ode} before normalizing the partial derivatives in \eqref{eq:rmsprop:parameter-ode} by dividing by the square root of the magnitudes of the filter states $\estimate{\gradientmagnitudesquared}_i$. The challenge in this work is that, for a model-free RMSprop ESC (RMSpESC), we need to somehow construct a perturbation-based estimation of the squared partial derivatives. 

The answer to the issue to use two low-pass filters in the continuous-time analog of RMSprop. The differential equations are defined element-wise as
\begin{subequations}
	\label{eq:rmsprop:model-free}
\begin{align}
	\label{eq:rmsprop:model-free:parameter-ode}
	\derivative{t} \estimate{\parameter}_i &{}={} \frac{-\gain\estimate{\gradient}_i(\ditherrate t, \estimate{\parameter}, \ditheramplitude)}{\sqrt{\vert\estimate{\gradientmagnitudesquared}_i\vert + \epsilon}} \\
	\label{eq:rmsprop:model-free:squared-gradient-estimate-ode}
	\derivative{t} \estimate{\gradientmagnitudesquared}_i &{}={} \omega_{\gradientmagnitudesquared} \left(\estimate{\altgradient}_i^2 - \estimate{\gradientmagnitudesquared}_i\right) \\
	\label{eq:rmsprop:model-free:gradient-filter}
	\derivative{t}\estimate{\altgradient}_i &= \omega_{\gradient}\left(\estimate{\gradient}_i\left(\omega t, \estimate{\parameter}, a\right) - \estimate{\altgradient}_i\right)
\end{align}
\end{subequations}
for $i=1,2,\ldots,n$ with $\estimate{\altgradient}$ a low pass filter of the perturbation-based gradient estimate $\estimate{\gradient}(\ditherrate, \estimate{\parameter}, \ditheramplitude)$ defined as
\begin{equation}
	\estimate{\gradient}\left(\omega t, \estimate{\parameter}, a\right) = \gradientfilter(\omega t, a)\costfunction\left(\estimate{\parameter} + a \dithersignal(\omega t)\right)
\end{equation}
Here, $\gradientfilter$ is a time-varying signal, dependant on the dither amplitude $\ditheramplitude$, used to demodulate the gradient from the fastly perturbed cost function.The exact $\gradientfilter$ signal is a design decision left up to the system designer but the following assumption is made to satisfy the main results of \cite{ref:james-kavanaugh-2026}.
\begin{asmp}
	\label{asmp:consistent-gradient-filter}
	For a given $\dithersignal$, the gradient demodulation signal $\gradientfilter:\realnumbers\times\strictlypositivereals\to\realnumbers^n$ is periodic or almost periodic in the first argument and satisfies both
	\begin{equation}
		\lim_{T\to\infty} \frac{1}{T} \int_{0}^T m(\tau, \ditheramplitude) d\tau = 0
	\end{equation}
	and
	\begin{equation}
		\lim_{T\to\infty} \frac{1}{T} \int_{0}^{T} \gradientfilter_i(\tau, \ditheramplitude) \dithersignal_j(\tau) d\tau = 
		\begin{cases}
			\ditheramplitude^{-1} & \text{ if } i = j \\
			0 & \text{ if } i \neq j
		\end{cases}
	\end{equation}
\end{asmp}
The above assumption is not difficult to satisfy with traditional sinusoidal signals, with $\gradientfilter$ often being defined element-wise as $\gradientfilter_i(\ditherrate t, \ditheramplitude) = \frac{2}{\ditheramplitude\relativeamplitude_i}\sin(\relativerate_i \ditherrate t)$.

With the system design choosing approapriate $\dithersignal$ and $\gradientfilter$ such that Assumptions~\ref{asmp:dither-signal-linearly-independent} and \ref{asmp:consistent-gradient-filter} are satisfied, the following the stability results can be given in Theorem~\ref{thm:rmsprop-is-sgpuas} with the relavent stability definitions found in Appendix~\ref{app:practical-stability}.
\begin{thm}
	\label{thm:rmsprop-is-sgpuas}
	The system, described in \eqref{eq:rmsprop:model-free:parameter-ode}-\eqref{eq:rmsprop:model-free:gradient-filter} with Assumptions~\ref{asmp:dither-signal-linearly-independent}-\ref{asmp:consistent-gradient-filter} satisfied, is sGPUAS to $(\optimal{\parameter}, 0, 0)$ with respect to the small parameters $(\ditheramplitude, \ditherrate^{-1})$.
\end{thm}
The difference between Theorem~\ref{thm:rmsprop-is-sgpuas} and those found in the framework approaches, e.g.~\cite{ref:nesic-2012}, is the number of small parameters required to state practical stability. Specifically, it is not necessary to include $\omega_{\gradient}^{-1}$ as a small parameter which eliminates an intermediate timescale given in \cite{ref:nesic-2012} as $\omega_{\gradient} t$. This allows for selections of $\omega$ to be smaller than in the framework approaches when attempting to meet practical convergence, namely the region of attraction and the size of the final convergence neighborhood. In addition, with a few additional restrictions, it is possible to eliminate the dither amplitude $\ditheramplitude$ as a small parameter for practical stability.
\begin{thm}
	\label{thm:rmsprop-is-sgpuas:strictly-convex-scalars}
	Let $\costfunction:\realnumbers\to\realnumbers$ be twice differentiable cost function which is strictly convex, i.e., for any distinct pair $\parameter_1,\parameter_2\in\realnumbers$, the cost function satsifies
	\begin{equation}
		\costfunction(\lambda\parameter_1 + (1-\lambda)\parameter_2) < \lambda \costfunction(\parameter_1) + (1-\lambda)\costfunction(\parameter_2)
	\end{equation}
	for all $\lambda\in(0,1)$. Additionally, define the dither and gradient demodulation signal as $\dithersignal(\tau) = \sin(\tau)$ and $\gradientfilter(\tau,\ditheramplitude) = \frac{2}{\ditheramplitude}\sin(\tau)$ respectively. If Assumption~\ref{asmp:unique-minimum} is satisfied, then there exists a point $\stationary{\parameter}$ such that $\vert \stationary{\parameter} - \optimal{\parameter} \vert \leq \ditheramplitude$ where the system in \eqref{eq:rmsprop:model-free:parameter-ode}-\eqref{eq:rmsprop:model-free:gradient-filter} is sGPUAS to $(\stationary{\parameter}, 0, 0)$ with respect to the small parameter $\ditherrate^{-1}$.
\end{thm}
	
\section{Semiglobal Practical Stability Proofs}
	\label{sec:sgpuas-proofs}

Practical stability analysis generally does not analyze the model-free system in \eqref{eq:rmsprop:model-free:parameter-ode}-\eqref{eq:rmsprop:model-free:gradient-filter} directly, but instead relies on either one or two auxiliary system. The first auxiliary system, which is necessary for both the proofs of Theorem~\ref{thm:rmsprop-is-sgpuas} and \ref{thm:rmsprop-is-sgpuas:strictly-convex-scalars}, is the average system. The average system is found in the asymptotic limit of $\ditherrate\to\infty$ and defined as
\begin{subequations}
	\label{eq:rmsprop:average-system}
\begin{align}
	\label{eq:rmsprop:average-system:parameter-ode}
	\derivative{t} \avgest{\parameter}_i &{}={} \frac{-\gain\avgest{\gradient}_i \left(\avgest{\parameter}, \ditheramplitude\right)}{\sqrt{\vert\avgest{\gradientmagnitudesquared}_i\vert + \epsilon}} \\
	\label{eq:rmsprop:average-system:squared-gradient-estimate-ode}
	\derivative{t} \avgest{\gradientmagnitudesquared}_i &{}={} \omega_{\gradientmagnitudesquared} \left(\avgest{\altgradient}_i^2 - \avgest{\gradientmagnitudesquared}_i\right) \\
	\label{eq:rmsprop:average-system:gradient-filter}
	\derivative{t}\avgest{\altgradient}_i &= \omega_{\gradient}\left(\avgest{\gradient}_i\left(\avgest{\parameter}, a\right) - \avgest{\altgradient}_i\right)
\end{align}
\end{subequations}
where
\begin{equation}
	\label{eq:average-gradient-estimate}
	\avgest{\gradient}\left(\avgest{\parameter}, \ditheramplitude\right)  =  \lim_{T\to\infty} \frac{1}{T}\int_0^T \estimate{\gradient}\left(\tau, \avgest{\parameter}, \ditheramplitude\right) d\tau \,.
\end{equation}
is the time-average of $\estimate{\gradient}$. The existence of the average system is well-defined whenever $\dithersignal$ and $\gradientfilter$ are almost periodic and $\costfunction\in\continuous^1$, all of which is assumed in Assumptions~\ref{asmp:dither-signal-linearly-independent}, \ref{asmp:continuously-differentiable}, and \ref{asmp:consistent-gradient-filter}. More formally, we need to show that the vector field of the model-free system is locally Lipschitz with respect to state vector $\estimate{\state} = [\estimate{\parameter}^\top, \estimate{\gradientmagnitudesquared}^\top, \estimate{\altgradient}^\top]^\top \in \realnumbers^{3n}$ and that it is almost periodic with respect to $\ditherrate t$. The local Lipschitz property can be demonstrated by Assumption~\ref{asmp:continuously-differentiable} \cite[Ch~2.2]{ref:cobzas-2019} and the fact that locally Lipschitz functions are closed necessary algebraic operations \cite[Ch~2.3]{ref:cobzas-2019} to form the vector field. Similarily, almost periodic functions are closed under the necessary albegra operations \cite[Th~1.9]{ref:fink-1975}. Therefore, the model-free RMSpESC is a Krylov, Bogoliubov, and Mitropolsky (KBM) vector field \cite[Lem~4.6.4]{ref:sanders-2007}, which is a sufficient condition for the existence of the average system \cite[Def~4.2.4]{ref:sanders-2007}.

The second system is the model-based system, which is given in the sequential limit of a vanishing dither amplitude ($\ditheramplitude\to 0$) and only used for the proof of Theorem~\ref{thm:rmsprop-is-sgpuas}. In the limit $\ditheramplitude\to 0$, we have $\avgest{\gradient}\left(\altparameter, \ditheramplitude\right) \to \nabla\costfunction\left(\altparameter\right)$ by Assumptions~\ref{asmp:dither-signal-linearly-independent} and \ref{asmp:consistent-gradient-filter} with \cite[Th~1]{ref:james-kavanaugh-2026}. Thus, the model-based system is
\begin{subequations}
	\label{eq:rmsprop:model-based}
	\begin{align}
		\label{eq:rmsprop:model-based:parameter-ode}
		\derivative{t}\altparameter_i 
			&= \frac{
					-\gain\frac{\partial\costfunction
					}{
					\partial\altparameter_i
					}\left(\altparameter\right)
				}{
					\sqrt{\vert\altgradientmagnitudesquared_i\vert + \varepsilon}
				} \\
		\label{eq:rmsprop:model-based:squared-gradient-estimate-ode}
		\derivative{t}\altgradientmagnitudesquared_i 
			&= \omega_{\gradientmagnitudesquared}\left(
				\altgradient_i^2 - \altgradientmagnitudesquared_i
			\right) \\
		\label{eq:rmsprop:model-based:gradient-filter}
		\derivative{t} \altgradient_i 
			&= \omega_{\altgradient}\left(
					\frac{\partial\costfunction}{\partial\altparameter_i}(\altparameter)
					- \altgradient_i
				\right) 
	\end{align}
\end{subequations}
It is notable that the model-free system in \eqref{eq:rmsprop:model-based} does not exactly match the continuous-time RMSprop \eqref{eq:rmsprop:ct-analogue}. This is due to the fact that we cannot directly estimate the squared gradient in a perturbation-based framework.

With the average and model-based systems defined, we can begin with the proof of the Theorem~\ref{thm:rmsprop-is-sgpuas}.

\begin{proof}[of Theorem~\ref{thm:rmsprop-is-sgpuas}]
	A sufficient condition for the model-free system to be sGPUAS wrt the small parameters $(\ditheramplitude, \ditherrate^{-1})$ is for the model-based system to be GAS. To accomplish this, we will take advantage of the following properties of the model-based system for constructing a Lyapunov function:
	\begin{itemize}
		\item[\namedlabel{itm:j-strictly-decreasing}{P1})]
			The derivative of $J$ along the trajectories of the model-based system is strictly decreasing for all $\altparameter \neq 0$.
		\item[\namedlabel{itm:iss-subsystems}{P2})]
			The $\altgradientmagnitudesquared$- and $\altgradient$-subsystems are input-to-state stable.
		\item[\namedlabel{itm:cyclically-connected}{P3})]
			The model-based system is a cyclically interconnected subsystems.
	\end{itemize}
	Property \ref{itm:j-strictly-decreasing} is easy to verify since the derivative of $\costfunction$ along the trajectories of the model-based system is
	\begin{equation}
		\label{eq:proof:cost-function-nonincreasing}
		\derivative[\costfunction]{t}(\altparameter, \altgradientmagnitudesquared) 
			= -\gain\left(\sum_{i=1}^n  \frac{
					\frac{\partial\costfunction
					}{
					\partial\altparameter_i
					}\left(\altparameter\right)^2
				}{
					\sqrt{\vert\altgradientmagnitudesquared_i\vert + \varepsilon}
				}
			\right) \leq 0
	\end{equation}
	with equality to zero if and only if $\altparameter = \optimal{\parameter}$ by Assumption~\ref{asmp:minimum-is-unique-stationary-point}. Notably, $\derivative[\costfunction]{t}$ depends on $\altgradientmagnitudesquared$ but $\altgradientmagnitudesquared$ only affects the rate at which $\costfunction$ decreases along the trajectories, it can never change the sign of $\derivative[\costfunction]{t}$.
	
	The ISS properties are also straight forward to verify. If we take $ \Vert \altgradientmagnitudesquared \Vert^2$ and $\Vert\altgradient\Vert^2$ to be the ISS-Lyapunov functions for the $\altgradientmagnitudesquared$- and $\altgradient$-subsystems, their derivatives along the model-based system trajectories satisfy
	\begin{subequations}
		\label{eq:proof:subsystem-lyapunov-derivative}
	\begin{align}
		\label{eq:proof:subsystem-lyapunov-derivative:gradient-squared-lpf}
		\derivative[\Vert\altgradientmagnitudesquared\Vert^2]{t} 
			&\leq -2\omega_{\gradientmagnitudesquared} \left(
				\Vert\altgradientmagnitudesquared\Vert^2
				-\Vert\altgradientmagnitudesquared\Vert\ \sqrt{\sum_{i=1}^n \altgradient_i^4}
				\right) \\
		\label{eq:proof:subsystem-lyapunov-derivative:gradient-lpf}
		\derivative[\Vert\altgradient\Vert^2]{t}
			&\leq -2\omega_{\gradient} \left(
				\Vert \altgradient \Vert^2 
				- \Vert\altgradient \Vert\ 
					\left\Vert
						\nabla\costfunction(\altparameter) 
					\right\Vert
				\right)
	\end{align}
	\end{subequations}
	These differential inequalities can be rewritten to give implication-form ISS inequalities of
	\begin{subequations}
	\begin{align}
		\label{eq:proof:subsystem-implication-iss-inequalities:gradient-squared-lpf}
		\derivative[\Vert\altgradientmagnitudesquared\Vert^2]{t}
			&\leq -\omega_{\gradientmagnitudesquared} \Vert\altgradientmagnitudesquared\Vert^2 
				& \forall\ \Vert\altgradientmagnitudesquared \Vert 
				&\geq \rho_{\gradientmagnitudesquared}(\Vert\altgradient\Vert) \\
		\label{eq:proof:subsystem-implication-iss-inequalities:gradient-lpf}
		\derivative[\Vert\altgradient\Vert^2]{t} 
			&\leq -\omega_{\gradient} \Vert\altgradient\Vert^2 
				& \forall\ \Vert \altgradient \Vert 
				&\geq \rho_{\altgradient}\left(\Vert\altparameter\Vert\right)
	\end{align}
	\end{subequations}
	where $\rho_{\gradientmagnitudesquared}(r) = 2 r^2$ and $\rho_{\altgradient}$ is aclass $\classK$ functions satisfying
	\begin{equation}
		\left\Vert \nabla \costfunction (\parameter)\right\Vert \leq \frac{1}{2} \rho_{\altgradient}(\Vert \parameter \Vert) \,.
	\end{equation}
	We note that the step from \eqref{eq:proof:subsystem-lyapunov-derivative:gradient-lpf} to \eqref{eq:proof:subsystem-implication-iss-inequalities:gradient-lpf} utilizes the bounds
	\begin{equation}
		\sqrt{\sum_{i=1}^n \altgradient_i^4} = \Vert \altgradient \Vert_4^{2} \leq \Vert \altgradient \Vert_2^{2}
	\end{equation}
	for $L^p$ norms. The existence of $\rho_{\altgradient}$ is guaranteed by the continuity of $\nabla\costfunction$, a property which is derived from Assumption~\ref{asmp:continuously-differentiable}. We are now ready to construct a Lyapunov function for the model-based system based on the ideas of the max-form Lyapunov function in \cite{ref:wang-2026} for adaptive gradient methods. 
	
	The start of the Lyapunov construction begins by attempting to determine boundedness of the model-free system by analyzing the loop of the interconnected system begining from the $\altparameter$-system. The cost function, as noted in Remark~\ref{rem:cost-function-bounded}, is both radially unbounded and globally descecrant. More importantly, $\derivative[\costfunction]{t}$ is negative so the level sets of $\costfunction$ are forward invariant and therefore the trajectories of the $\altparameter$-subsystem are bounded. Therefore, for some initial condition $\altparameter(t_0) = \altparameter_0$, all trajectories of $\altparameter$ satisfy
	\begin{equation}
		\Vert \altparameter(t) \Vert \leq  \alpha_{\parameter,1}^{-1}\left(\costfunction(\altparameter_0)\right)\quad\forall\ t\in[t_0,\infty)
	\end{equation}
	for \emph{any} possible $\altgradientmagnitudesquared(t)$ and $\altgradient(t)$ of the model-based system. The boundability of $\parameter(t)$ uniformly over possible $\altgradientmagnitudesquared(t)$ and $\altgradient(t)$ is crucial because $\altparameter$ acts at as an input to the $\altgradient$-subsystem through the partial derivatives of $\costfunction$. Consequently, if we rewrite \eqref{eq:proof:subsystem-implication-iss-inequalities:gradient-lpf} to
	\begin{align}
		\label{eq:proof:subsystem-implication-iss-inequalities:gradient-lpf:rewritten}
		\derivative[\Vert\altgradient\Vert^2]{t}
			&\leq -\omega_{\gradient} \Vert\altgradient\Vert^2 
				& \forall\ \Vert \altgradient \Vert 
				&\geq \rho_{\altgradient}\circ\alpha_{\parameter, 1}^{-1}\circ\costfunction\left(\altparameter\right)\,,
	\end{align}
	then we can give the bound for the trajectory $\altgradient$, for given initial conditions $\altgradient(t_0) = \altgradient_{0}$ and $\altparameter_0$, as
	\begin{align}
		\Vert \altgradient(t) \Vert \leq \max\left\lbrace \Vert\altgradient_{0}\Vert, \rho_{\altgradient}\circ\alpha_{\parameter, 1}^{-1}\circ\costfunction\left(\altparameter_0\right) \right\rbrace
	\end{align}
	for all $t\in[t_0,\infty)$. Thus, the aggregate $(\altparameter, \altgradient)$-system is bounded uniformly over the trajectories $\altgradientmagnitudesquared(t)$. To continue the analysis along the loop, we can easily show that the individual $\altgradientmagnitudesquared_i$ trajectories can be bounded from the initial conditions $\altparameter_0$, $\altgradient_{0}$, and $\altgradientmagnitudesquared_{i}(t_0) = \altgradientmagnitudesquared_{0}$ by
	\begin{equation}
		\Vert \altgradientmagnitudesquared(t) \Vert 
			\leq \max\left\lbrace 
				\Vert \altgradientmagnitudesquared_{0} \Vert,
				\rho_{\gradientmagnitudesquared}(\Vert\altgradient_{0}\Vert),
				\rho_{\gradientmagnitudesquared}
					\circ \rho_{\altgradient}
					\circ\alpha_{\parameter, 1}^{-1}
					\circ\costfunction\left(\altparameter_0\right) 
				\right\rbrace
	\end{equation}
	
	At this point, we can conclude the model-based system \eqref{eq:rmsprop:model-based} has bounded trajectories. This is earlier than when boundedness can be concluded normally in cyclically connected interconnected systems. Usually, the system analyst would hanve to complete the loop back to the $\altparameter$-subsystem and see if the loop gain is contractive (cf.~\cite{ref:jiang-1996}), but we have already shown that is not necessary for the model-based system. The $\altparameter$-subsystem directly influences the boundedness of the $\altgradient$-subsystem and indirectedly influences the bounds of the $\altgradientmagnitudesquared$-subsystems through the $\altgradient$-subsystem. However, this influence is only in one direction which is a property to be exploited in our construction. Specifically, the explotation comes by considering the Lyapunov function
	\begin{multline}
		\label{eq:proof:lyapunov-function}
		\lyapunov\left(\altparameter, \altgradientmagnitudesquared, \altgradient\right)
			= \costfunction\left(\altparameter\right)  \\
				+ \max\left\lbrace
					\gamma_{\gradientmagnitudesquared,\parameter}
						\circ\costfunction(\altparameter),
					\gamma_{\gradientmagnitudesquared,\gradient}
						\left(\Vert\altgradient\Vert^2\right),
					\Vert \altgradientmagnitudesquared \Vert^2
				\right\rbrace
	\end{multline}
	where
	\begin{equation*}
		\gamma_{\gradientmagnitudesquared,\parameter} (r)
			= r^2 
				\circ \rho_{\gradientmagnitudesquared} 
				\circ \rho_{\gradient}
				\circ \alpha_{\parameter,1}^{-1}
				(r)
	\end{equation*}
	and ${
		\gamma_{\altgradientmagnitudesquared,\gradient}(r) 
			= r^2 
				\circ \rho_{\gradientmagnitudesquared} 
				\circ\sqrt{r}
			= 4r^2
	}$ are class $\classK$ functions representing the influence of the $\altparameter$- and $\altgradient$-subsystems influence on the $\altgradientmagnitudesquared$-subsystem. Using the $\max$ function identities $\max\lbrace r_1 ,r_2, r_3\rbrace \geq (r_1 + r_2 + r_3)/3$ and $\max\lbrace r_1, r_2, r_3\rbrace \leq r_1 + r_2 + r_3$ for $r_1,r_2,r_3\geq 0$, $\lyapunov$ is both radially unbounded and has both a lower bound
	\begin{equation}
		\label{eq:proof:lyapunov-lower-bound}
		\lyapunov(\altparameter, \altgradientmagnitudesquared, \altgradient) 
			\geq 
				\alpha_{\parameter, 1}( \Vert \altparameter \Vert )  
				+ \frac{4}{3} \Vert \altgradient \Vert^4
				+ \frac{1}{3} \Vert \altgradientmagnitudesquared \Vert^2
	\end{equation}
	that is radially unbounded and an upper-bound 
	\begin{multline}
		\label{eq:proof:lyapunov-upper-bound}
		\lyapunov(\altparameter, \altgradientmagnitudesquared, \altgradient) 
			\geq 
				\alpha_{\parameter, 2}( \Vert \altparameter \Vert )  
				+ \gamma_{\gradientmagnitudesquared,\parameter}
					\circ\alpha_{\parameter, 2}( \Vert \altparameter \Vert ) \\
				+ 4 \Vert \altgradient \Vert^4
				+ \Vert \altgradientmagnitudesquared \Vert^2
	\end{multline}
	that is continuous. By \cite[Sec~24]{ref:hahn-1967}, these two bounds implies that there exists $\alpha_1,\alpha_2 \in\classKinfty$ such that
	\begin{equation}
		\alpha_1(\Vert \altstate \Vert) 
			\leq V(\altparameter, \altgradientmagnitudesquared, \altgradient) 
				\leq \alpha_2( \Vert \altstate \Vert)
	\end{equation}
	where $\altstate = [\altparameter^\top, \altgradientmagnitudesquared^\top, \altgradient^\top]^\top$ is the state vector of the model-based system.	Note that while the $\classKinfty$ bounds are implicity, one explicit possible $\alpha_2$ is
	\begin{equation}
		\alpha_2(r) 
			= \alpha_{\parameter, 2}(r)
				+ \gamma_{\gradientmagnitudesquared,\parameter}
					\circ\alpha_{\parameter, 2}(r)
				+ 4 r^4 + r^2
	\end{equation}
	but obtaining an explicit $\alpha_1$ is harder and likely requires some additional knowledge about $\costfunction$ for a readially given expression.
	
	With a radially unbounded and globally descecrent $\lyapunov$ obtained, we now seek to prove that that this $\lyapunov$ is strictly decreasing along the trajectories of the model-free system. The upper-right Dini derivative of $\lyapunov$ along the model-free system is
	\begin{multline}
		\upperrightdini \lyapunov(
			\altparameter, \altgradientmagnitudesquared, \altgradient
		) 
		= \limsup_{h\to 0^+} \frac{1}{h}\left[
			\lyapunov(
				\altparameter(t+h),
				\altgradientmagnitudesquared(t+h),
				\altgradient(t+h)
			) \right. \\ \left. - \lyapunov(
				\altparameter(t), 
				\altgradientmagnitudesquared(t), 
				\altgradient(t))
		\right] \, ,
	\end{multline}
	and we seek to show that this is a negative definite function. The Dini derivative distributes over sums so
	\begin{multline*}
		\upperrightdini\lyapunov(
			\altparameter, \altgradientmagnitudesquared, \altgradient
		)
			= \upperrightdini\costfunction(\altparameter) \\
				+ \upperrightdini\max\left\lbrace
					\gamma_{\gradientmagnitudesquared,\parameter}
						\circ\costfunction(\altparameter),
					4\Vert\altgradient\Vert^4,
					\Vert \altgradientmagnitudesquared \Vert^2
				\right\rbrace
	\end{multline*}
	Since the upper-right Dini derivative of a scalar function is the same as the derivative of the function along the trajectory when the function is differentiable, we know that $\upperrightdini\costfunction \equiv \derivative[\costfunction]{t} $ where $\derivative[\costfunction]{t}$ was given in \eqref{eq:proof:cost-function-nonincreasing}. The Dini derivative of the maximum is more work since it depends on the active argument(s) of the maximum, specifically, it will be the maximum of Dini derivative of all active arguments. We now go over the three possible active arguements.
	
	The first argument to consider is ${
	\gamma_{\gradientmagnitudesquared,\parameter}\circ\costfunction(\altparameter)
	}$ and this argument may not be differentiable due to $\alpha_{\parameter,1}^{-1}$ potentially not being differentiable. However, since $\costfunction$ is strictly nonincreasing along the trajectories by \eqref{eq:proof:cost-function-nonincreasing}, we can give the upper right Dini derivative as
	\begin{equation}
		\upperrightdini
				\gamma_{\gradientmagnitudesquared, \parameter}
				\circ\costfunction(\altparameter)
			= \derivative[\costfunction]{\altparameter}\left(
					\altparameter, \altgradientmagnitudesquared
				\right)\left(
					\lowerleftdini\gamma_{\gradientmagnitudesquared, \parameter}(r)
				\right)\circ\costfunction(\altparameter)
	\end{equation}
	where
	\begin{equation*}
		\lowerleftdini\gamma_{\gradientmagnitudesquared, \parameter}(r) 
			= \liminf_{h\to 0^+} \frac{1}{h}\left[
					\gamma_{\gradientmagnitudesquared, \parameter}(r) 
					- \gamma_{\gradientmagnitudesquared, \parameter}(r-h)
				\right]
	\end{equation*}
	is the lower left Dini derivative of ${
		\gamma_{\gradientmagnitudesquared, \parameter}
	}$. Since ${
		\gamma_{\gradientmagnitudesquared, \parameter}
	}$ is a class $\classK$ function, ${
		\lowerleftdini\gamma_{\gradientmagnitudesquared, \parameter}(r) \geq 0
	}$ by the strict monotonicity of ${
		\gamma_{\gradientmagnitudesquared, \parameter}
	}$. Therefore, we conclude that
	\begin{equation}
		\upperrightdini
				\gamma_{\gradientmagnitudesquared, \parameter}
				\circ\costfunction(\altparameter)
			\leq 0 \,.
	\end{equation}
	The next active argument to consider is ${
		\gamma_{\gradientmagnitudesquared,\gradient}
			\left(\Vert\altgradient\Vert^2\right)
	}$. If this argument is active, then the condition in \eqref{eq:proof:subsystem-implication-iss-inequalities:gradient-lpf:rewritten} is meet and 
	\begin{align}
		\upperrightdini \gamma_{\altgradientmagnitudesquared,\altgradient}
			&= 
				8 \Vert\altgradient\Vert^2 \;
					\derivative[\Vert\altgradient\Vert^2]{t}
			\leq - 8 \omega_{\gradient} \Vert \altgradient\Vert^4 
			\eqqcolon -W_{\gradient} (\altgradient)
	\end{align}
	The last active argument to consider is ${
		\Vert\altgradientmagnitudesquared\Vert^2
	}$. This argument being active implies that the condition in \eqref{eq:proof:subsystem-implication-iss-inequalities:gradient-squared-lpf} is meet. Consequently, 
	\begin{equation}
		\upperrightdini \Vert\altgradientmagnitudesquared\Vert^2
			= \derivative[\Vert\altgradientmagnitudesquared\Vert^2]{t}
			\leq 
				-\omega_{\gradientmagnitudesquared} \Vert
					\altgradientmagnitudesquared
				\Vert^2 
			\eqqcolon -W_{\gradientmagnitudesquared}(\altgradientmagnitudesquared)
	\end{equation}
	whenever $\Vert\altgradientmagnitudesquared\Vert^2$ is an active argument in the maximum. Note that in all possible active arguments, the upper right Dini derivative of the maximum with be less than or equal to zero. Combining all the possible active arguments to the maximum results in
	\begin{equation}
		\label{eq:proof:lyapunov-derivative:upper-semicontinuous-bound}
		\upperrightdini\lyapunov 
			\leq \derivative[\costfunction]{t} 
				- \begin{cases}
					0 & 
						\text{ if A1} \\
					\min\lbrace
						W_{\gradient}(\altgradient),
						W_{\gradientmagnitudesquared}(\altgradientmagnitudesquared)
					\rbrace &
						\text{ both A2 and A3 } \\
					W_{\gradient}(\altgradient) &
						\text{ only A2} \\
						W_{\gradientmagnitudesquared}(\altgradientmagnitudesquared) &
						\text{ only A3}
				\end{cases}
	\end{equation}
	where A1, A2, and A3 indicate what argument is active in the maximum of $\lyapunov$.
	
	The bound on $\upperrightdini\lyapunov$ given in \eqref{eq:proof:lyapunov-derivative:upper-semicontinuous-bound} is a negative definite, upper semicontinuous function. This is enough for us to conclude that $\lyapunov$ is nonincreasing along the trajectories of the model-based system in \eqref{eq:rmsprop:model-based}. Thus, the model-based system is both bounded and stable since, for any $c > 0$, we have
	\begin{equation}
		\Vert \altstate(t_0) \Vert \leq c \implies \Vert \altstate(t) \vert \leq \alpha_1^{-1} \circ \alpha_2 (c) \ \forall\ t\in[t_0,\infty)
	\end{equation}
	and
	\begin{equation}
		\Vert \altstate(t_0) \Vert \leq \alpha_2^{-1}\circ\alpha_1(c) \implies \Vert \altstate(t) \vert \leq c \ \forall\ t\in[t_0,\infty)
	\end{equation}
	Additionally, the system is uniformly asymptotically stable. This can be seen by a proof by contradiction in the same manner as \cite[Th~8.2]{ref:verhulst-1996}. Since the system has already been shown to be bounded and stable, then there must be some annulus for which a trajectory with an initial condition inside of the annulus that stays forever within this annulus for the system not to be asymptotically stable. Assume that such a trajectories exists. Its initial condition must satisfy ${\Vert \altstate(t_0) \Vert \leq c_1}$ for some $c_1 > 0$ and it never enters a ball of radius $c_2 > 0$. Without loss of generality, assume that $c_2 < c_1$. Note that if $\lyapunov(\altstate(t_1)) \leq \alpha_1(c_2)$ at any point $t_1\in[t_0,\infty)$, then the trajectorie must be within the ball of radius $c_2$ for all $t\in[t_1,\infty)$ as the level sets of $\lyapunov$ are forward invariant and the level set less than or equal to $\alpha_1(c_2)$ is a subset of the aforementioned ball. Let $\mu$ be defined as
	\begin{subequations}
	\begin{align}
		\mu = &\min_{\altstate\in\realnumbers^{3}} W(z) \\
		\subto \phantom{=} & \alpha_2^{-1}\circ\alpha_1(c_2) \leq \Vert \altstate \Vert \leq \alpha_1^{-1}\circ\alpha_2(c_1)
	\end{align}
	\end{subequations}
	where $-W(\altstate)$ is the RHS of \eqref{eq:proof:lyapunov-derivative:upper-semicontinuous-bound}. This leads to the differential inequality $\upperrightdini\lyapunov \leq -\mu $ for all $\lyapunov \in [\alpha_1(c_2), \alpha_1(c_1)]$. By the Comparison Lemma \cite[Th~3.4]{ref:khalil-2002} and the forward-invariance of the level sets of $\lyapunov$, all trajectories starting within the a ball of radius $c_1$ must have a $\lyapunov(\altstate(t)) \leq \alpha_1(c_2)$ after $T(c_1, c_2) = T$ seconds, where
	\begin{equation}
		T(c_1, c_2) = \frac{\alpha_2(c_1) - \alpha_1(c_2)}{\mu}
	\end{equation}	 
	Since $c_1$ and $c_2$ were arbitrary, we conclude that
	\begin{equation}
		\Vert \altstate(t_0) \Vert \leq c_1 \implies \Vert \altstate(t) \Vert \leq c_2 \ \forall\ t\in[t_0 + T, \infty)
	\end{equation}
	This contradicts the system not being asymptotically stable. Therefore, the model-based system is stable, bounded, and uniformly asymptotically stable. These are the three parts of the definition of GAS, so the model-based system is GAS.

	To continue the proof that the model-free system is sGPUAS, we could invoke \cite[Th~1]{ref:mcnamee-2025} where the model-based system being GAS was a sufficient condition for the model-free system being sGPUAS. However, \cite[Th~1]{ref:mcnamee-2025} used a converse Lyapunov theorem that a Lyapunov function for the model-based system which was differentiable but we would like to proceed with the $\lyapunov$ we have constructed in order to use it for further analysis, such as the calculations of $\ditheramplitude^*$ using level sets of $\lyapunov$. Our $\lyapunov$ is not necessarily differentiable due to the potential nondifferentiablility of $\gamma_{\gradientmagnitudesquared, \parameter}$. This fact does not disuade us as what matters is the differentiability of $\costfunction$ and the ISS-Lyapunov functions for the auxiliary variables. For the average system, we have
	\begin{subequations}
	\begin{align}
		\derivative[\costfunction\left(\avgest{\parameter}\right)]{t}
			&= \begin{multlined}[0.6\columnwidth][t]
				-\gain\left(\sum_{i=1}^n  \frac{
						\frac{\partial\costfunction
						}{
							\partial\avgest{\parameter}_i
						}\left(\avgest{\parameter}\right)^2 
					}{
					\sqrt{\vert\altgradientmagnitudesquared_i\vert + \varepsilon}
				}
			\right) \\
			-\gain\left(\sum_{i=1}^n  \frac{
						R_{\gradient,i}\left(
							\avgest{\parameter}, \ditheramplitude
						\right)
						\frac{\partial\costfunction
						}{
							\partial\avgest{\parameter}_i
						}\left(\avgest{\parameter}\right)
					}{
					\sqrt{\vert\avgest{\gradientmagnitudesquared}_i\vert + \varepsilon}
				}
			\right)
			\end{multlined} \\
		\derivative[\Vert \avgest{\altgradient} \Vert ^2]{t}
			&\leq \begin{multlined}[0.6\columnwidth][t]
				-\omega_{\gradient} \Vert \avgest{\altgradient} \Vert^2 
				+ 2\omega_{\gradient}\left\Vert
					\avgest{\altgradient}
				\right\Vert \left\Vert 
					R_{\gradient}\left(\avgest{\parameter}, \ditheramplitude\right) 
				\right\Vert \\ 
				\forall\ \avgest{\altgradient}_i^2 
					\geq \gamma_{\altparameter, \altgradient_i}
						\circ\costfunction\left(\avgest{\parameter}\right)
				\end{multlined}
	\end{align}
	\end{subequations}
	where $R_{\gradient,i}\left(\avgest{\parameter}, \ditheramplitude\right)$ is the $i$th element of the remainder
	\begin{equation}
		R_{\gradient}\left(\avgest{\parameter}, \ditheramplitude\right)
			= \avgest{\gradient}\left(\avgest{\parameter}, \ditheramplitude\right)
				-  \nabla\costfunction  \left(\avgest{\parameter}\right)
	\end{equation}
	Reference \cite{ref:james-kavanaugh-2026} showed that $R_{\gradient}\left(\avgest{\parameter},\ditheramplitude\right)\to 0$ pointwise as $\ditheramplitude\to 0$ with the convergence rate the same as $\nabla\costfunction(\parameter)\to\nabla\costfunction\left(\avgest{\parameter}\right)$ when $\parameter\to\avgest{\parameter}$ but this pointwise converge is actually the stronger convergence of compact convergence, i.e., local uniform convergence over compact subsets of the domain \cite{ref:mcnamee-2024}. Thus, for any compact set $\domain\subset\realnumbers^{3n}$ and $\varepsilon > 0$, we can always find an $\critical{\delta}(\varepsilon)$ such that $\left\Vert R_{\gradient}\left(\avgest{\parameter},\ditheramplitude \right) \right\Vert < \varepsilon$ for all $\ditheramplitude\in(0,\critical{\delta})$ and $\avgest{\parameter}\in\domain$. Important to the proof that the average system is practically stable, we consider the annulus
	\begin{equation}
		\annulus_{\rho_1,\rho_2} 
			= \left\lbrace 
				\avgest{\state}\ \vert\ 
				\rho_1 \leq \Vert \avgest{\state} \Vert \leq \rho_2 \right\rbrace 
	\end{equation}
	with radii $0 < \rho_1 < \rho_2 < \infty$ partitioned into subsets
	\begin{align*}
		\annulus_1 
			&= \left\lbrace 
				\avgest{\state} \in\annulus_{\rho_1,\rho_2}\ \vert\ 
					\gamma_{\gradientmagnitudesquared,\parameter}
						\circ\costfunction(\avgest{\parameter}) 
						\geq \max\lbrace 
							4\Vert \avgest{\gradient} \Vert^4, 
							\Vert\avgest{\gradientmagnitudesquared}\Vert^2 
						\rbrace 
				\right\rbrace \\
		\annulus_2
			&= \left\lbrace 
				\avgest{\state} \in\annulus_{\rho_1,\rho_2}\ \vert\ 
					4\Vert \avgest{\gradient} \Vert^4 
						\geq \max\lbrace 
							\gamma_{\gradientmagnitudesquared,\parameter}
								\circ\costfunction(\avgest{\parameter}), 
							\Vert\avgest{\gradientmagnitudesquared}\Vert^2 
						\rbrace 
				\right\rbrace \\
		\annulus_3
			&= \left\lbrace 
				\avgest{\state} \in\annulus_{\rho_1,\rho_2}\ \vert\ 
					\Vert\avgest{\gradientmagnitudesquared}\Vert^2
						\geq \max\lbrace 
							\gamma_{\gradientmagnitudesquared,\parameter}
								\circ\costfunction(\avgest{\parameter}), 
							4\Vert \avgest{\gradient} \Vert^4 	 
						\rbrace 
				\right\rbrace
	\end{align*}
	based on the active arguments of the maximum in $\lyapunov$. The compact convergence property of $R_\gradient$ means that there exists an $\critical{\ditheramplitude}$ such that, for all $\ditheramplitude\in(0,\critical{\ditheramplitude})$, the following bounds hold:
	\begin{equation}
		\left\vert
			\sum_{i=1}^n  
				\frac{
					R_{\gradient,i}\left(
						\avgest{\parameter}, \ditheramplitude
					\right)
					\frac{
						\partial\costfunction
					}{
						\partial\avgest{\parameter}_i
					}\left(\avgest{\parameter}\right)
				}{
					\sqrt{\vert\avgest{\gradientmagnitudesquared}_i\vert + \varepsilon}
				}
		\right\vert 
			\leq (1 - c)
				\sum_{i=1}^n  
					\frac{
						\frac{
							\partial\costfunction
						}{
							\partial\avgest{\parameter}_i
						}\left(\avgest{\parameter}\right)^2
					}{
						\sqrt{
							\vert\avgest{\gradientmagnitudesquared}_i\vert + \varepsilon
						}
					}
	\end{equation}
	for all $\avgest{\state}\in\annulus_1$;
	\begin{multline}
		\left\vert
			\sum_{i=1}^n  
				\frac{
					R_{\gradient,i}\left(
						\avgest{\parameter}, \ditheramplitude
					\right)
					\frac{
						\partial\costfunction
					}{
						\partial\avgest{\parameter}_i
					}\left(\avgest{\parameter}\right)
				}{
					\sqrt{\vert\avgest{\gradientmagnitudesquared}_i\vert + \varepsilon}
				}
			+ 16 \omega_\gradient \Vert\avgest{\gradient}\Vert^3 \Vert
				\Vert R_\gradient (\avgest{\parameter}, \ditheramplitude) \Vert
		\right\vert \\
			\leq 
				(1 - c)\left[\left(\sum_{i=1}^n  
						\frac{
							\frac{
								\partial\costfunction
							}{
								\partial\avgest{\parameter}_i
							}\left(\avgest{\parameter}\right)^2
						}{
							\sqrt{
								\vert\avgest{\gradientmagnitudesquared}_i\vert
								+ \varepsilon
							}
						}\right)
					+ 8 \omega_{\gradient} \Vert
						\avgest{\gradientmagnitudesquared}
					\Vert^4
				\right]
	\end{multline}
	for all $\avgest{\state}\in\annulus_2$; and
	\begin{equation}
		\left\vert
			\sum_{i=1}^n  
				\frac{
					R_{\gradient,i}\left(
						\avgest{\parameter}, \ditheramplitude
					\right)
					\frac{
						\partial\costfunction
					}{
						\partial\avgest{\parameter}_i
					}\left(\avgest{\parameter}\right)
				}{
					\sqrt{\vert\avgest{\gradientmagnitudesquared}_i\vert + \varepsilon}
				}	
		\right\vert 
			\leq 
				\begin{aligned}[t]
					&(1 - c)\sum_{i=1}^n  
						\frac{
							\frac{
								\partial\costfunction
							}{
								\partial\avgest{\parameter}_i
							}\left(\avgest{\parameter}\right)^2
						}{
							\sqrt{
								\vert\avgest{\gradientmagnitudesquared}_i\vert
								+ \varepsilon
							}
						} \\
					& + (1 - c) \omega_{\gradientmagnitudesquared} \Vert
						\avgest{\gradientmagnitudesquared}
					\Vert^2
				\end{aligned}
	\end{equation}
	for all $\avgest{\state}\in\annulus_3$. Therefore, we can conclude that there is an $\critical{\ditheramplitude}(\rho_1, \rho_2, c) = \critical{\ditheramplitude}$ such that, for all $\ditheramplitude\in(0,\critical{\ditheramplitude})$, the upper right Dini derivative of $\lyapunov$ along the average system is bounded by
	\begin{equation}
		\upperrightdini\lyapunov\left(
			\avgest{\parameter}, \avgest{\gradientmagnitudesquared}, \avgest{\gradient}
		\right) 
			\leq - c W \left(
				\avgest{\parameter}, 
				\avgest{\gradientmagnitudesquared}, 
				\avgest{\gradient}
			\right) 
	\end{equation}
	for all $\avgest{\state}\in\annulus_{\rho_1,\rho_2}$. This the same as \cite[Rem~3]{ref:mcnamee-2025}, so $\lyapunov$ and the compact convergence of $R_\gradient$ proves that the average system is sGPAS with respect to the small parameter $\ditheramplitude$ by the steps of Part~1 of \cite[Th~1]{ref:mcnamee-2025} and, subsequently, the model-free system is sGPUAS with respect to the small parameters ($\ditheramplitude, \ditherrate^{-1}$) by Part 2 of \cite[Th~1]{ref:mcnamee-2025}.
\end{proof}

While we have proven that the RMSpESC is practically stable for multivariable $\costfunction$, there is additional structure in the average system for $\costfunction$'s that are scalar, twice differentiable, and strictly convex which can be utilize to eliminate $\ditheramplitude$ for the small parameters of practical stability. 

\begin{proof}[of Theorem~\ref{thm:rmsprop-is-sgpuas:strictly-convex-scalars}]
	To emphasize that $\ditheramplitude$ is fixed, we will remove it as an argument in the following proof. The choice of ${\dithersignal(\tau)=\sin(\tau)}$ and ${\gradientfilter(\tau) = \frac{2}{\ditheramplitude}\sin(\tau)}$ results in $\avgest{\gradient}(\cdot)$ being a strictly monotonically increasing and having a unique zero $\stationary{\parameter}\in[\optimal{\parameter} - \ditheramplitude, \optimal{\parameter} + \ditheramplitude]$  \cite{ref:mcnamee-2025}. For convenience, we denote an error coordinate system of $\avgerror{\parameter} = \avgest{\parameter} - \stationary{\parameter}$. The monotinicity of $\avgest{\gradient}$ allows us to determine that
	\begin{equation}
		\derivative[\avgerror{\parameter}^2]{t}
			= - \frac{
					2 \gain \avgerror{\parameter}\; \avgest{\gradient}\left(
						\avgerror{\parameter} + \stationary{\parameter}
					\right)}{
					\sqrt{
						\vert \avgest{\gradientmagnitudesquared} \vert + \varepsilon
					}
				} \leq 0
	\end{equation}
	is a negative definite function that is equal to zero if and only if $\avgerror{\parameter} = 0$. With this insight, we will proceed to construct a Lyapunov function for the average system which similar to the one for the model-based system but using $\avgerror{\parameter}^2$ in place of $\costfunction$. As such, we need to bound the growth of $\avgest{\gradient}$ from $\stationary{\parameter}$. We do know that, by the continuity of $\avgest{\gradient}$, there exists a $\rho_{\avgest{\gradient}}\in\classKinfty$ such that
	\begin{equation}
		\left\vert 
			\avgest{\gradient}\left(
				\avgerror{\parameter} + \stationary{\parameter}
			\right)
		\right\vert 
			\leq \rho_{\avgest{\gradient}}\left(
				\left\vert \avgerror{\parameter} \right\vert
			\right)
	\end{equation}
	This bound allows us to form the Lyapunov function
	\begin{equation}
		\lyapunov\left(
			\avgerror{\parameter}, 
			\avgest{\gradientmagnitudesquared},
			\avgest{\altgradient}
		\right) = 
			\avgerror{\parameter}^2 + \max \left\lbrace
				\gamma_{\avgest{\gradientmagnitudesquared}, \avgerror{\parameter}}
					\left(\avgerror{\parameter}^2\right),
				4 \avgest{\altgradient}^4,
				\avgest{\gradientmagnitudesquared}^2
			\right\rbrace
	\end{equation}
	with $\gamma_{\avgest{\gradientmagnitudesquared}, \avgerror{\parameter}} = (r^2)\circ\rho_{\avgest{\gradient}}\circ(\sqrt{r})$. This Lyapunov function is structurely identical to the Lyapunov function presented in \eqref{eq:proof:lyapunov-function} so, by identical arguments those found in the previous proof, the Lyapunov function is radially unbounded and its derivative along the trajectories of the average system is bounded by
	\begin{multline}
		 \derivative[\lyapunov]{t} 
		 \leq 
		 	- \frac{
			 	2\gain\avgerror{\parameter}\; 
			 	\avgest{\gradient}\left(
			 		\avgerror{\parameter} + \stationary{\parameter}
			 	\right)
			 }{
			 	\sqrt{
			 		\left\vert \avgest{\gradientmagnitudesquared} \right\vert 
			 		+ \varepsilon
			 	}
			 } \\
			 - 
			 \begin{cases}
			 	0 & \text{ if A1} \\
			 	\min\lbrace 
			 			W_\gradient(\avgest{\altgradient}), 
			 			W_\gradientmagnitudesquared (\avgest{\gradientmagnitudesquared})
			 		\rbrace & \text{ both A2 and A3} \\
			 	W_\gradient(\avgest{\altgradient}) & \text{ only A2} \\
			 	W_\gradientmagnitudesquared (
			 			\avgest{\gradientmagnitudesquared}
			 		) & \text{ only A3}
			 \end{cases}
	\end{multline}
	which is a upper semicontinuous, negative definite function. Therefore, the average system is GAS \cite[Th~4.9]{ref:khalil-2002} and the model-free system is sGPUAS to $(\stationary{\parameter}, 0, 0)$ with respect to the small parameter $\ditherrate^{-1}$ \cite{ref:teel-1998}.
\end{proof}
	
\section{Numerical Examples}

For comparison of the RMSpESC to other ESCs, we will use two ESCs from literature; the gradient-based ESC (GESC) and the Newton-based ESC (NESC) \cite{ref:ghaffari-2012}. For ease of reference, we will give them here. The GESC is defined as
\begin{equation}
	\label{eq:gesc}
	\derivative{t} \estimate{\parameter} 
		= - \gain \estimate{\gradient}\left(
			\ditherrate t, \estimate{\parameter}, \ditheramplitude
		\right)
\end{equation}
while the NESC is defined by
\begin{subequations}
	\label{eq:nesc}
	\begin{align}
		\derivative{t}\estimate{\parameter} 
			&= -\gain\estimate{\invhessian}\estimate{\gradient}\left(
					\ditherrate t, \estimate{\parameter}, \ditheramplitude 
				\right )\\
		\label{eq:nesc:inverse-hessian}
		\derivative{t}\estimate{\invhessian}
			&= \omega_{\invhessian}\left(
					\estimate{\invhessian} 
					- \estimate{\invhessian}\ \estimate{\hessian}\left(
						\ditherrate t, \estimate{\parameter}, \ditheramplitude
					\right)\ \estimate{\invhessian}
				\right)
	\end{align}
\end{subequations}
with
\begin{equation}
	\estimate{\hessian}\left(
		\tau, \estimate{\parameter}, \ditheramplitude
	\right) = \hessianfilter\left(
		\tau, \ditheramplitude
	\right)\costfunction\left(
		\estimate{\parameter} + \ditheramplitude\dithersignal\left(\tau\right)
	\right)
\end{equation}
the perturbation-based estimate of the Hessian as demodulated from the perturbed cost function by $\hessianfilter$. For classical sinusoidal $\dithersignal$, $\hessianfilter$ can be given as
\begin{equation}
	\label{eq:hessian-filter}
	\hessianfilter_{ij}\left(\tau, \ditheramplitude\right) = \begin{cases}
		-\frac{8}{\ditheramplitude^2 \relativeamplitude_i^2} \cos\left(
				2\relativerate_i \tau
			\right) & \text{ for } i=j \\
		\frac{4}{\ditheramplitude^2 \relativeamplitude_i\relativeamplitude_j}
			\sin\left(\relativerate_i \tau\right) \sin\left(\relativerate_j \tau\right)
			& \text{ for } i \neq j
	\end{cases}
\end{equation}
For reasons similar to $\avgest{\gradient}$, the average Hessian estimate
\begin{equation}
	\avgest{\hessian}\left(\avgest{\parameter},\ditheramplitude\right) 
		= \lim_{T\to\infty} \frac{1}{T}\int_0^T 
			\estimate{\hessian}\left(\tau, \avgest{\parameter}, \ditheramplitude \right)
		d\tau
\end{equation}
converges compactly to the Hessian $\nabla^2 \costfunction$ as $\ditheramplitude\to 0$ \cite{ref:james-kavanaugh-2026}.

\subsection{Quadratic Example}
	\label{sec:numerical-examples:quadratic}

The simpliest cost function for study is the multivariable quadratic cost function. In this section, we will consider the cost function
\begin{equation}
	\costfunction\left(\parameter\right)
		= \frac{1}{2} \parameter^\top \begin{bmatrix}
			\frac{1}{2} & 0 \\
			0 & 10
		\end{bmatrix} \parameter
\end{equation}
where the Hessian has a conditioning number of 20 and the minimizer is $\optimal{\parameter} = 0$. Some representative trajectories can be seen in Figure~\ref{fig:multivariable-quadratic}. The model-free GESC trajectory exhibits behavior typical of the model-based GESC for poorly conditioned problems; the trajectories converges to the line associated with the smallest eigenvalue ($\estimate{\parameter}_1$ here) before travaling along this line towards the minimizer. The speed of convergence to the minimizer is dominated by the smallest eigenvalue of the Hessian ($1/2$ in this problem) so the system designer wants to compensate for potentially small eigenvalues in situations where $\costfunction$ is not known ahead of time. The NESC is one such ESC design which would compensate for eigenvalues by trying to estimate the inverse of the Hessian $\estimate{\invhessian}$ in \eqref{eq:nesc:inverse-hessian}, leading to trajectories that go straight to the minimizer after some transience dynamics. The RMSpESC lies somewhere in between these two behaviors for this example because the RMSpESC (and RMSprop) attempt to normalize the speed of $\estimate{\parameter}$ in each component direction.

\begin{figure*}[t!]
	\centering
	\includegraphics[width=\columnwidth]{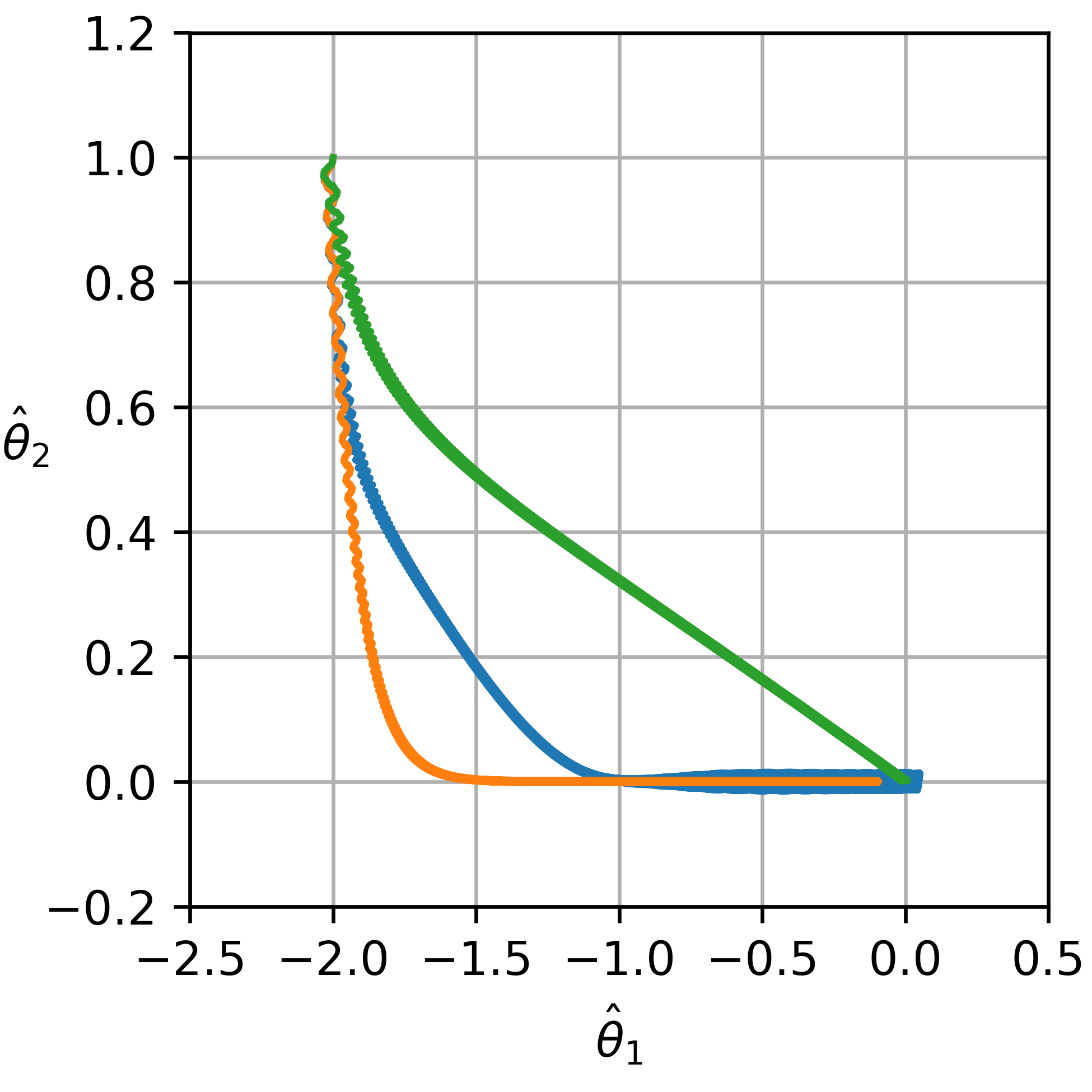}
	\hfill
	\includegraphics[width=\columnwidth]{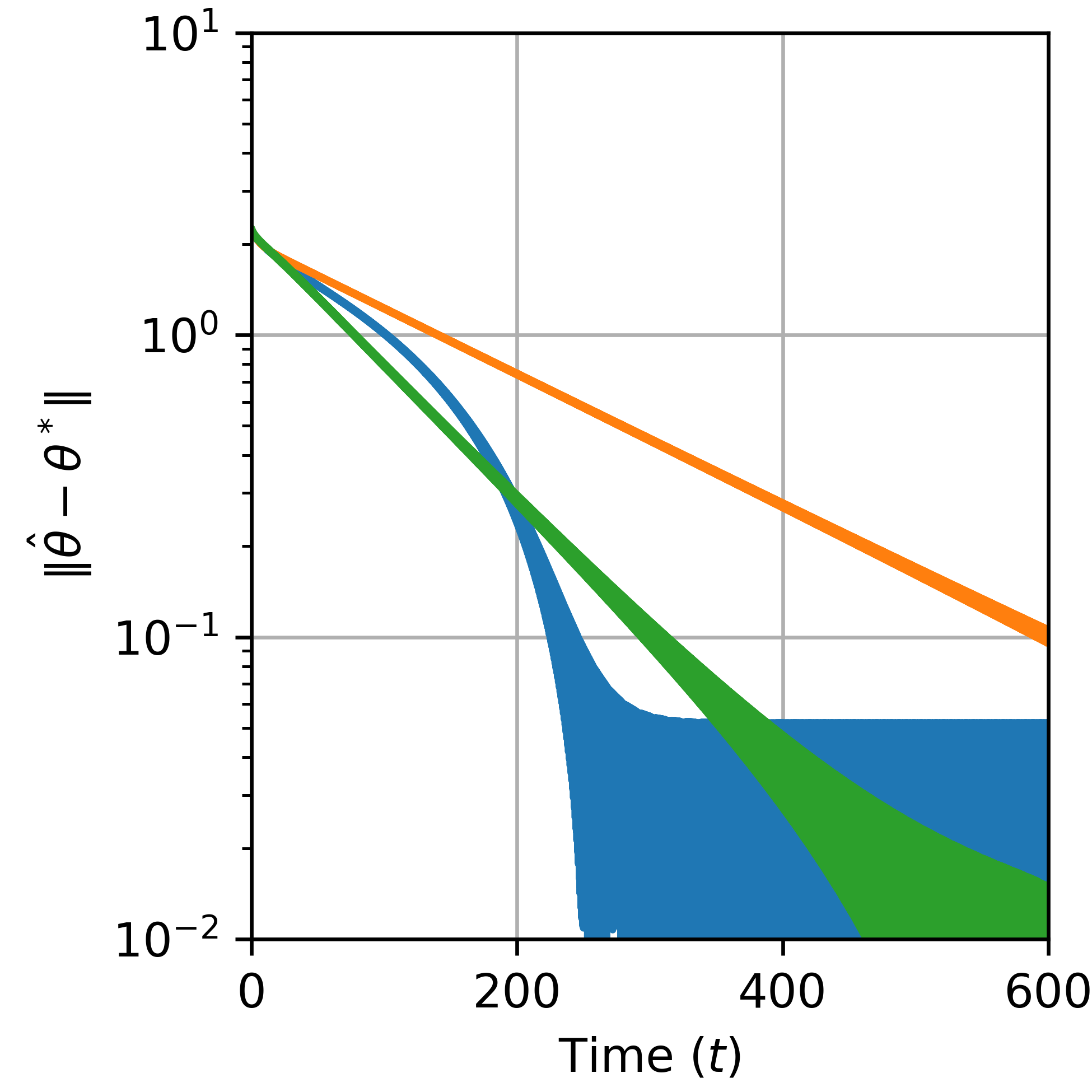}
	\caption{Results of the RMSpECS, GESC, and NESC acting on the multivariable quadratic cost function in blue, orange, and green lines respectively with trajectories shown on the left and the distance from the optimum over time shown on the right. All trajectories start from $\estimate{\parameter}_0 = [-2, 1]^\top$ with $\relativerate_1 = 1$, $\relativerate_2 = 5$, $\ditherrate=10$, $\gain=0.01$, and $\ditheramplitude=\relativeamplitude_1 = \relativeamplitude_2 = 1$. Additionally, RMSpESC used the parameter values $\omega_{\gradientmagnitudesquared} = \omega_{\gradient}= 0.1$ and $\varepsilon = 0.01$ while NESC used $\omega_{\invhessian}=0.06$. The other relavent initial conditions for RMSpESC and NESC were $\estimate{\gradientmagnitudesquared}_0 = [0.99, 0.99]^\top$, $\estimate{\altgradient}_0 = 0$, and $\estimate{\invhessian}_0 = \identitymatrix$.}
	\label{fig:multivariable-quadratic}
\end{figure*}

Trajectories in $\estimate{\parameter}$ do not tell the entire story. The time history in Figure~\ref{fig:multivariable-quadratic} shows both the RMSpESC and the NESC converging to their final asymptotic bounds faster than the standard GESC. Notably, the RMSpESC reaches its final asymptotic bound the fastest, although its bound is the largest due to $\varepsilon$ in \eqref{eq:rmsprop:model-free:parameter-ode} amplifying the vector field when $\estimate{\gradientmagnitudesquared}$ is small. Clearly, RMSpESC's regularization allows it have more control over the convergence speed to its final asymptotic bound compared to GESC and, consequently, converge faster in its example. However, it is not apparent in Figure~\ref{fig:multivariable-quadratic} why the NESC is not the most desirable ESC, converging faster than the GESC and having a smaller final asymptotic bound than RMSpESC. 

The answer lies in the experience of the system designer, who knows the tempermental nature of the NESC. Figure~\ref{fig:multivariable-quadratic} uses a $\ditherrate/\gain$ ratio of $10^3$, with regions of attractions growing and final asymptotic bounds shrinking as this ratio increases. However, this ratio also plays a role for the existence of finite escape time trajectories for considered initial conditions. If the ratio was reduced to $\ditherrate/\gain = 10^2$, then the GESC and RMSpESC would still have trajectories from the initial conditions that existed for all time but NESC would not. Adjusting other parameters yields similar results, the NESC has finite escape times if $\ditheramplitude$ was reduce to $1/2$ or if $\omega_{\invhessian}$ was increased to $0.1$ like $\omega_\gradientmagnitudesquared$ and $\omega_\gradient$. The robustness to changes of parameters of GESC and RMSpESC compared to NESC means that both the GESC and RMSpESC tend to be preferrable for deployment to real systems when $\costfunction$ is unknown \emph{a priori}.

Another consideration are the choices of both $\dithersignal$ and $\gradientfilter$, which were deliberately left vague in Assumptions~\ref{asmp:dither-signal-linearly-independent} and \ref{asmp:consistent-gradient-filter}. Most of the previous literature simply uses sines or cosines at different relative rates but perhaps there are different choices that could improve certain qualities of model-free ESCs. For instance, consider the dither signal and gradient demodulation signal used to generate Figure~\ref{fig:multivariable-quadratic} as $\dithersignal_0$ and $\gradientfilter_0$ respectively, with the zero subscript indicating them as a null hypothesis of sorts. There is nothing in the literature to suggest that the pairing of $(\dithersignal_0,\gradientfilter_0)$ is, in any way, the ``best'' pairing for some particular objective, such as achieving small final bounds for a given $\ditheramplitude$ or a larger region of attraction for fixed $(\ditheramplitude,\ditherrate)$, for some family of $\costfunction$. The system designer can consider experimenting with other signals, such as
\begin{equation}
	\dithersignal_1(\tau) 
		= \frac{1}{2}\dithersignal_0(\tau) 
			+ \frac{1}{2}\begin{bmatrix}
				\relativeamplitude_1 \sin(\sqrt{2}\tau) \\ 
				\relativeamplitude_2  \sin(2\tau)
			\end{bmatrix}
\end{equation}
or
\begin{equation}
	\gradientfilter_1(\tau, \ditheramplitude) 
		= \gradientfilter_0(\tau, \ditheramplitude) 
			+ \frac{2}{\ditheramplitude} \begin{bmatrix}
				\sin(\sqrt{2}\tau)/\relativeamplitude_1 \\
				\sin(\sqrt{11}\tau)/\relativeamplitude_2
			\end{bmatrix} \,,
\end{equation}
which may improve the integral convergence in \eqref{eq:average-gradient-estimate}. Alternative, the system designer can choose $\gradientfilter_2$ with its first element
\begin{multline}
	\gradientfilter_{2,1}(\tau, \ditheramplitude)
		= \frac{2}{\ditheramplitude\relativeamplitude_1}
			\left(
				\sin(\tau) + 3\sin(3\tau) - 3\sin(7\tau) \right. \\ \left. - \sin(9\tau) + \sin(11\tau) 
			\right)
\end{multline}
and its second element
\begin{equation}
	\gradientfilter_{2,2}(\tau, \ditheramplitude)
		= \frac{2}{\ditheramplitude\relativeamplitude_2}
			\left(\sin(5\tau) + 2\sin(7\tau) + 3\sin(15\tau)\right)
\end{equation}
This signal is designed based on the work provided in \cite{ref:james-kavanaugh-2026} to remove the effects of the third and fourth order derivatives on $\avgest{\gradient}$ for the particular pair $(\relativerate_1,\relativerate_2) = (1,5)$. The selected trajectories of pairings $(\dithersignal_0, \gradientfilter_0)$, $(\dithersignal_1, 2\gradientfilter_0)$, $(\dithersignal_0, \gradientfilter_1)$, and $(\dithersignal_0, \gradientfilter_2)$ are shown in Figure~\ref{fig:multivariable-quadratic:signal-comparison}. All the trajectories show roughly the same path shape as the RMSpESC in Figure~\ref{fig:multivariable-quadratic} but the size of the final asymptotic bound is different. Even through we can say from Theorem~\ref{thm:rmsprop-is-sgpuas} that each final asymptotic bound vanishes as $\ditherrate\to \infty$, we cannot say anything about the rate at which they vanish at this time.

\begin{figure}[t!]
	\centering
	\includegraphics[width=\columnwidth]{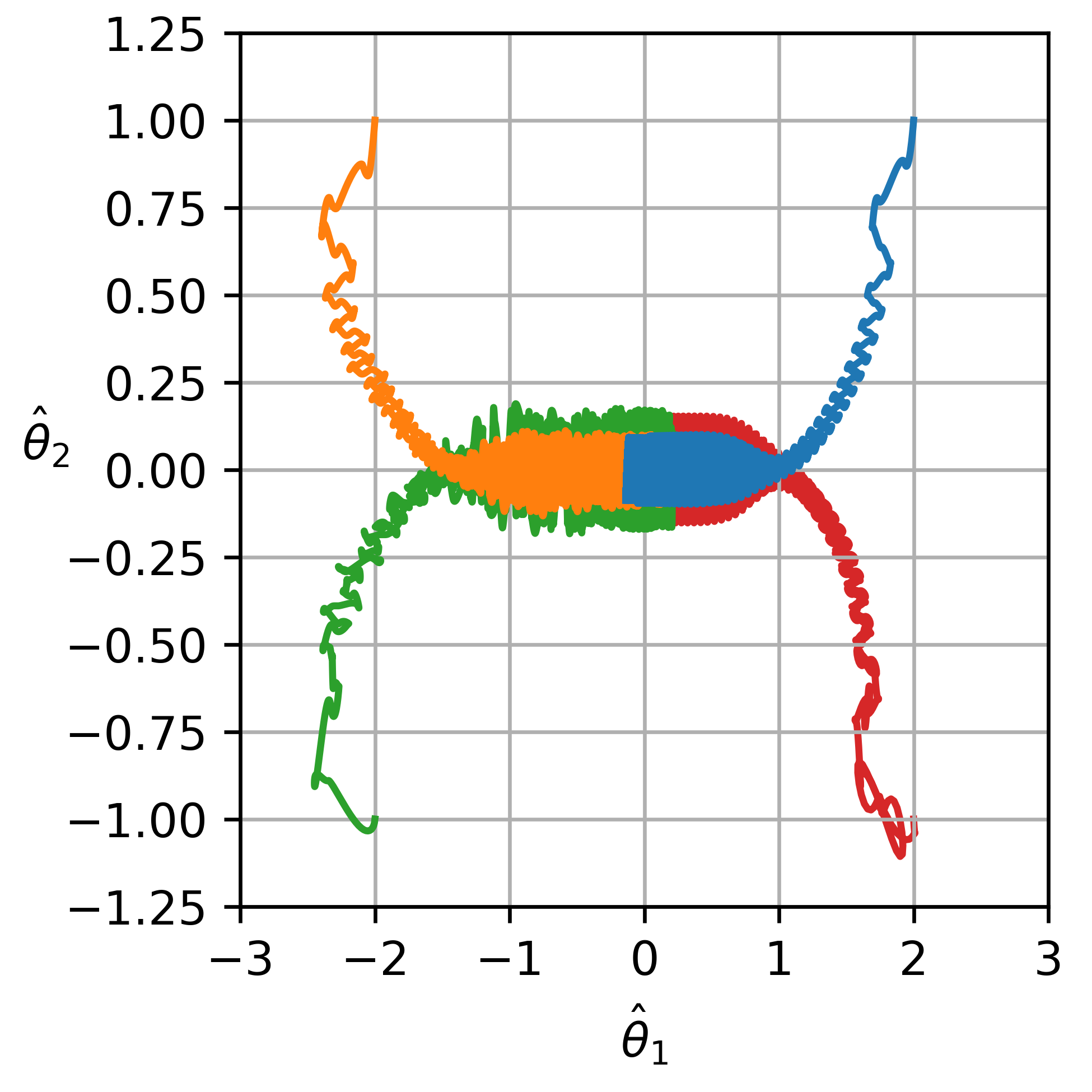}
	\caption{Various RMSpESC trajectories using different $(\dithersignal,\gradientfilter)$ pairs using the same parameters as in Figure~\ref{fig:multivariable-quadratic} except $\ditherrate=1$. The pairs $(\dithersignal_0, \gradientfilter_0)$, $(\dithersignal_1, 2\gradientfilter_0)$, $(\dithersignal_0, \gradientfilter_1)$, and $(\dithersignal_2, \gradientfilter_2)$ had initial conditions $\estimate{\parameter}_0$ of $[2,1]^\top$, $[-2,1]^\top$, $[-2,-1]^\top$, and $[2,-1]^\top$ respectively.}
	\label{fig:multivariable-quadratic:signal-comparison}
\end{figure}

Nevertheless, it is important to note that our results in Theorem~\ref{thm:rmsprop-is-sgpuas} apply to all of these combinations as the model-free ESCs and model-based ESCs share a many-to-one relationship. This many-to-one relationship is true in general, not just for this quadratic $\costfunction$ example. Actually, model-free ESCs and their average systems also share a many-to-one relationship. An example of this fact is that the pairs $(\dithersignal_0, \gradientfilter_0)$ and $(\dithersignal_0, \gradientfilter_1)$ always have the same identical average systems, regardless of the considered $\costfunction\in\continuous^0$, since we know
\begin{equation}
	\lim_{T\to\infty} \frac{1}{T} \int_0^T \begin{bmatrix}
	\sin(\sqrt{2}\tau) \\ \sin(\sqrt{11}\tau)
\end{bmatrix}	 \costfunction\left(\avgest{\parameter} + \ditheramplitude\dithersignal_0(\tau)\right) d\tau = 0
\end{equation}
from the results presented in \cite[Ch~7]{ref:sanders-2007}. In light to this many-to-one relationship between model-free ESCs and model-based ESCs, we acknowledge our result sGPUAS for RMSpESCs applies broadly over the family generated by pairs of $(\dithersignal, \gradientfilter)$ which satisfy our assumptions, but our result does not specify qualitative properties a system analyst may be interested in. These quantities may properties such as the size of the final asymptotic bound or the region of attraction as a function of $\ditheramplitude$ and $\ditherrate$; our result offers no information to function other than the region of attraction grows infinitely large and the final asymptotic bound vanishes in the limits of $\ditheramplitude\to 0$ and $\ditherrate\to \infty$. However, such a result over a wide family of extremum seekers is still valuable and allows for the system analyst to investigate $(\dithersignal,\gradientfilter)$ for their performances without fear that they are not practically stable.

\subsection{Strictly Convex Example}
	\label{sec:numerical-examples:strictly-convex}

We will consider an example from \cite{ref:mcnamee-2024} and \cite{ref:mcnamee-2025}, where the cost function,
\begin{equation}
	\costfunction(\parameter) = \parameter_1^4 + \left(\parameter_1 + \parameter_2\right)^4 \,,
\end{equation}
is strictly convex function with a minimizer of $\optimal{\parameter} = 0$. It has been shown that the average system is, at most, practically stable with respect to $\ditheramplitude$, with average systems being unstable at $\avgest{\parameter} = 0$ for fixed $\ditheramplitude$ and the sinusoidal dither signal with $(\relativerate_1, \relativerate_2) = (1,3)$ and $(\relativeamplitude_1, \relativeamplitude_2) = (12/\sqrt{145}, 1/\sqrt{145})$. Trajectories of the model-free GESC tend towards limit cycles whose size depends on $\ditheramplitude$. An advantage that RMSpESC has over GESC is the attempted normalization of the velocity of $\estimate{\parameter}$, which, as seen in Figure~\ref{fig:quartic-example}, tends to diminish the amplitude of any limit cycles the systems may have. Note the we have chosen not to show the NESC here since $\nabla^2 \costfunction(\optimal{\parameter}) = 0$ ensures that the model-based NESC is not GAS so it cannot be sGPUAS with respect to the small parameters $(\ditherrate^{-1}, \ditheramplitude)$.

\begin{figure}[t!]
	\centering
	\includegraphics[width=\columnwidth]{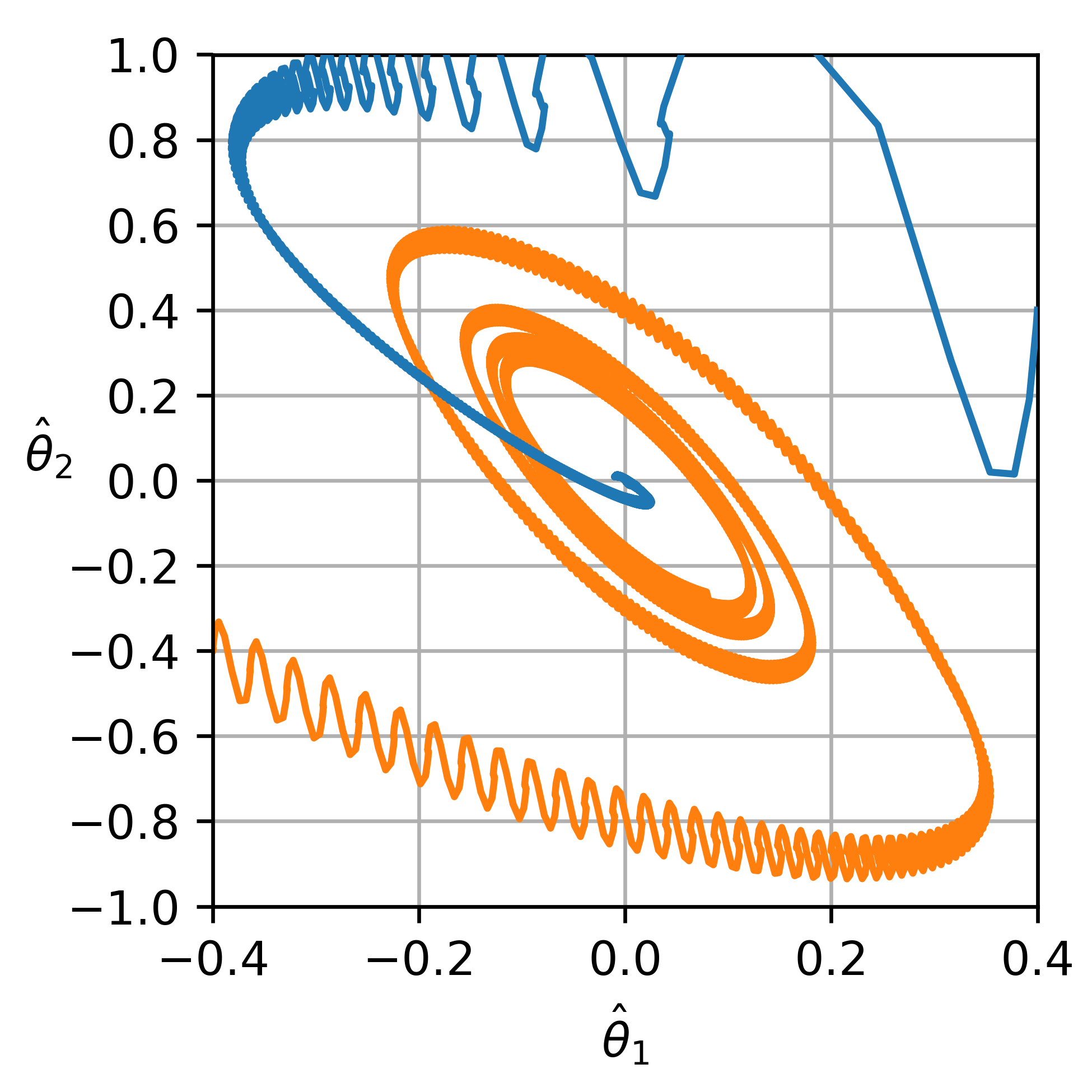}
	\caption{Comparison of the GESC and RMSpESC for a quartic cost function with . The parameters for the trajectories are $\ditherrate = 10$, $\gain=0.01$, $\omega_{\gradient} = 1$, $\omega_{\gradientmagnitudesquared}=0.01$, and $\varepsilon=0.01$. The GESC starts from $\estimate{\parameter}_0 = [-0.4, -0.4]^\top$ while the RMSpESC starts from $\estimate{\parameter}_0 = [0.4, 0.4]^\top$ with the other states initialized to $\estimate{\altgradient}_0 = \estimate{\gradientmagnitudesquared}=0$.}
	\label{fig:quartic-example}
\end{figure}

\subsection{Nonconvex Example}
	\label{sec:numerical-examples:nonconvex}

The last example covered in this section is the Rosenbrock function, defined as
\begin{equation}
	\costfunction(\parameter) 
		= 
			(1 - \parameter_1)^2 
			+ 100 \left(\parameter_2 - \parameter_1^2\right)^2 \,,
\end{equation}
which is a standard test cost function for optimization. The Rosenbrock function has a unique minimum at ${\optimal{\parameter} = [-1, 1]^\top}$ and is a nonconvex function. Model-based systems are known to generally have trajectories which tend to rapidly converge to the curve $\parameter_2 = \parameter_1^2$ before slowly traversing along the curve to $\optimal{\parameter}$. Trajectories of some model-free GESCs and RMSpESCs are shown in Figure~\ref{fig:rosenbrock} and these trajectories show some interesting behavior. First, the RMSpESC systems are slower to converge to their eventually limit cycles than their GESC counterparts which is the opposite of the results shown in Sections~\ref{sec:numerical-examples:quadratic} and \ref{sec:numerical-examples:strictly-convex}. This is not necessarily a detraction from the RMSpESC, since a system designer will generally choose a RMSpESC over a GESC for its regularization, even if it means that the speed of the seeker towards its eventual limit is slower.

The second behavior to note is improvement of the ESC when using $\gradientfilter_2$ over $\gradientfilter_0$. While the larger oscillations in $\estimate{\parameter}$ and Section~\ref{sec:numerical-examples:quadratic} may have lead one to believe that $\gradientfilter_0$ was always superior to $\gradientfilter_2$, this is not the case in this example. Every model-free ESC moves towards $\optimal{\parameter}$ from their initial point but the ESCs utilizing $\gradientfilter_0$ do not appreciably improve $\costfunction$ along their trajectories. In constrast, the ESCs which use $\gradientfilter_2$ for their gradient estimate drastically improve $\costfunction$ along their trajectories. This is because their trajectories following along the $\estimate{\parameter}_2 = \estimate{\parameter}_1^2$ curve. This is not necessarily a surprising result when considering $\gradientfilter_2$ corrects for up to fourth order derivatives when estimating $\estimate{\gradient}$. Since the Rosenbrock function is a quartic function, the average system is the model-based system when using $\gradientfilter_2$. Thus, we should expect the model-free system trajectories to more readily resemble the model-based system trajectories. However, neither ESCs utilizing $\gradientfilter_0$ or $\gradientfilter_2$ reach the $\optimal{\parameter}$ from the considered starting point and instead reach limit cycles away from $\optimal{\parameter}$. This is in contrast to most examples presented in the previous literature, with ESCs orbiting around (or at least very near) $\optimal{\parameter}$ which results in near optimal $\costfunction$ being obtained in any limiting periodic cycle. However, the results of Theorem~\ref{thm:rmsprop-is-sgpuas} still hold; the final aysmptotic bounds of the model-free ESC systems will shrink in the limits of $\ditheramplitude\to 0$ and $\ditherrate\to\infty$.

\begin{figure*}[t!]
	\centering
	\includegraphics[width=\columnwidth]{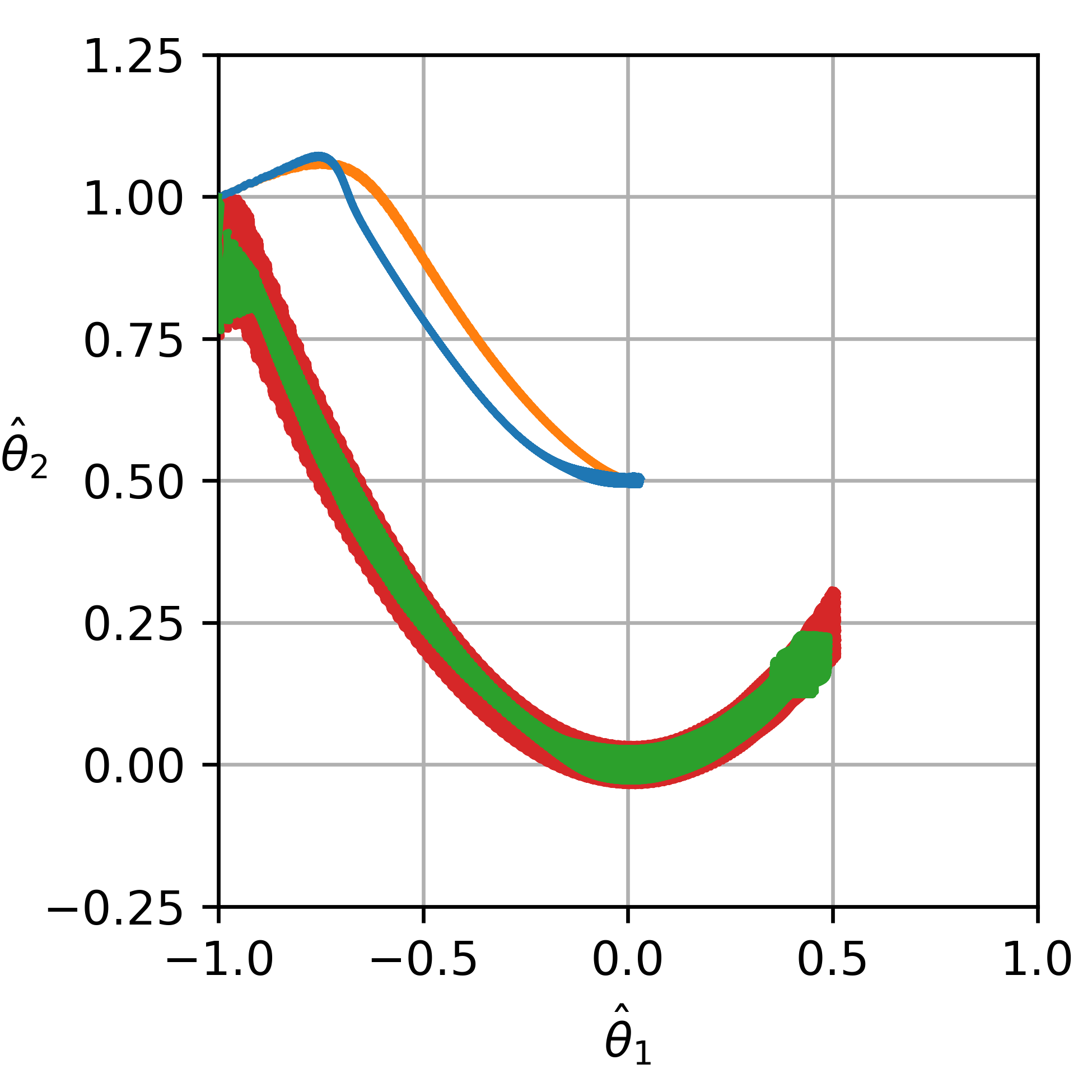}
	\hfill
	\includegraphics[width=\columnwidth]{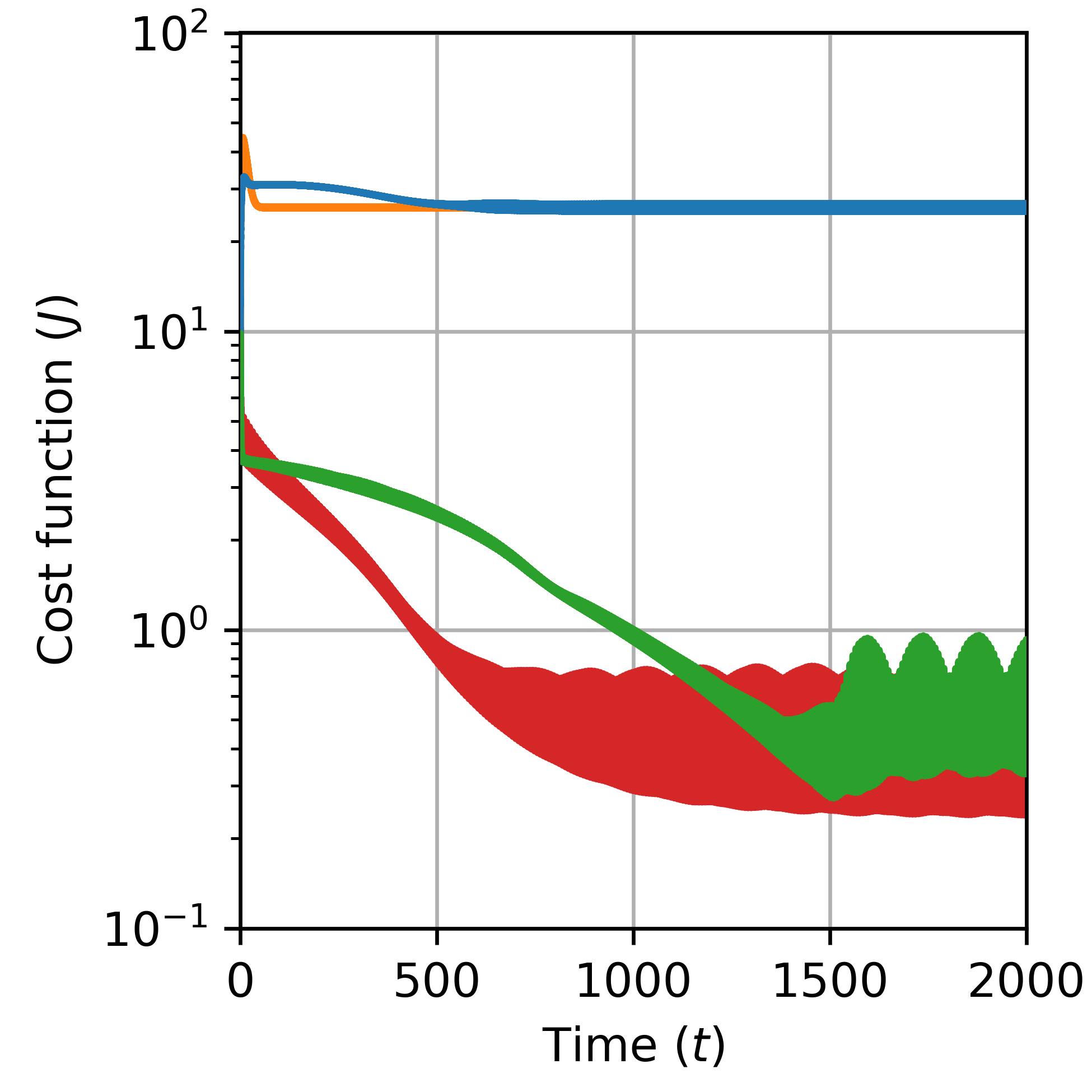}
	\caption{Trajectories of model-free GESCs and RMSpESCs shown on the left and the cost function along these trajectories shown on the right. Trajectories in orange and red are GESCs while those in blue and green are RMSpESCs. The blue and orange trajectories use the pair of signals $(\dithersignal_0,\gradientfilter_0)$ while the red and green trajectories use the pair pair of signals $(\dithersignal_0,\gradientfilter_2)$. All trajectories started from $\estimate{\parameter}_0 = [1, 1]^\top$ and used the parameter values $\ditheramplitude=1$, $\ditherrate=100$, and $\gain=0.001$. The RMSpESCs had the additional initial states of $\estimate{\gradientmagnitudesquared} = [0.99, 0.99]^\top$ and $\estimate{\altgradient}_0 = 0$ as well as the parameters $\omega_\gradientmagnitudesquared = \omega_gradient = 0.05$ and $\varepsilon = 0.01$.}
	\label{fig:rosenbrock}
\end{figure*}

\section{Conclusions}
In this work we have analyzed the novel RMSpESC, an extension of the RMSprop optimizer to  ESC. In particular, we have shown that the RMSpESC is sGPUAS for a wide variety of dither signals and gradient demodulation signals. Various numerical examples are shown to illustrate the various behaviors of the RMSpESC, such as its increased convergence speed over GESC in shallow regions and its robustness to parameter tuning over the NESC.

\bibliographystyle{IEEEtranS}
\bibliography{refs}

@misc{ref:wang-2026,
      title={Lyapunov Constructions for System Interconnections Arising from Adaptation in Some Optimization Methods}, 
      author={Qixu Wang and Patrick McNamee and Zahra Nili Ahmadabadi and Miroslav Krsti{\'c}},
      year={2026},
      eprint={2608.16851},
      archivePrefix={arXiv},
      primaryClass={math.OC},
      url={https://arxiv.org/abs/2608.16851}, 
}

@article{ref:james-kavanaugh-2026,
	author={James-Kavanaugh, Dylan and McNamee, Patrick and Wang, Qixu and Nili Ahmadabadi, Zahra},
	journal={IEEE Transactions on Control Systems Technology}, 
	title={Servos for Local Map Exploration Onboard Nonholonomic Vehicles for Extremum Seeking}, 
	year={2026},
	volume={34},
	number={5},
	pages={2733-2748},
	doi={10.1109/TCST.2026.3713366}
}

@article{ref:mcnamee-2025,
	title = {Stability of extremum seeking for maps that are strictly but not strongly convex},
	journal = {Automatica},
	volume = {182},
	pages = {112536},
	year = {2025},
	issn = {0005-1098},
	doi = {https://doi.org/10.1016/j.automatica.2025.112536},
	url = {https://www.sciencedirect.com/science/article/pii/S0005109825004315},
	author = {Patrick McNamee and Miroslav Krsti{\'c} and Zahra Nili Ahmadabadi}
}

@misc{ref:mcnamee-2024,
      title={Extremum Seeking is Stable for Scalar Maps that are Strictly but Not Strongly Convex}, 
      author={Patrick McNamee and Miroslav Krstić and Zahra Nili Ahmadabadi},
      year={2024},
      eprint={2405.12908},
      archivePrefix={arXiv},
      primaryClass={math.OC},
      url={https://arxiv.org/abs/2405.12908}, 
}

@InProceedings{ref:ma-2022,
  title = 	 {A Qualitative Study of the Dynamic Behavior for Adaptive Gradient Algorithms},
  author =       {Ma, Chao and Wu, Lei and E, Weinan},
  booktitle = 	 {Proceedings of the 2nd Mathematical and Scientific Machine Learning Conference},
  pages = 	 {671--692},
  year = 	 {2022},
  editor = 	 {Bruna, Joan and Hesthaven, Jan and Zdeborova, Lenka},
  volume = 	 {145},
  series = 	 {Proceedings of Machine Learning Research},
  month = 	 {16--19 Aug},
  publisher =    {PMLR},
  url = 	 {https://proceedings.mlr.press/v145/ma22a.html}
}

@BOOK{ref:cobzas-2019,
  title     = "Lipschitz functions",
  author    = "Cobza{\c s}, {\c S}tefan and Miculescu, Radu and Nicolae,
               Adriana",
  publisher = "Springer Nature",
  series    = "Lecture Notes in Mathematics",
  edition   =  2019,
  month     =  may,
  year      =  2019,
  address   = "Cham, Switzerland",
  language  = "en"
}

@inproceedings{ref:xiao-2012,
  title={Online speed optimization for sailing yachts using extremum seeking},
  author={Xiao, Lin and Alves, Jose C and Cruz, Nuno A and Jouffroy, Jerome},
  booktitle={2012 Oceans},
  pages={1--6},
  year={2012},
  organization={IEEE}
}

@article{ref:dincmen-2012,
  title={Extremum-seeking control of ABS braking in road vehicles with lateral force improvement},
  author={Din{\c{c}}men, Erkin and G{\"u}ven{\c{c}}, Bilin Aksun and Acarman, Tankut},
  journal={IEEE Transactions on Control Systems Technology},
  volume={22},
  number={1},
  pages={230--237},
  year={2012},
  publisher={IEEE}
}

@INPROCEEDINGS{ref:nesic-2012,
  author={Ne{\v{s}}i\'c, Dragan and Ying Tan and Manzie, Chris and Mohammadi, Alireza and Moase, Will},
  booktitle={2012 24th Chinese Control and Decision Conference (CCDC)}, 
  title={A unifying framework for analysis and design of extremum seeking controllers}, 
  year={2012},
  volume={},
  number={},
  pages={4274-4285},
  doi={10.1109/CCDC.2012.6244001}
}

@article{ref:ghaffari-2012,
	title = {Multivariable Newton-based extremum seeking},
	journal = {Automatica},
	volume = {48},
	number = {8},
	pages = {1759-1767},
	year = {2012},
	issn = {0005-1098},
	doi = {https://doi.org/10.1016/j.automatica.2012.05.059},
	author = {Azad Ghaffari and Miroslav Krsti\'c and Dragan Ne{\v{s}}i\'c}
}

@Book{ref:sanders-2007,
	author = {J.A. Sanders and F. Verhulst and J. Murdock},
	title = {Averaging Methods in Nonlinear Dynamical Systems},
	publisher = {Springer},
	year = {2007},
	edition = {2nd},
	month = {June}
}

@article{ref:popovic-2006,
  title={Extremum seeking methods for optimization of variable cam timing engine operation},
  author={Popovic, Dobrivoje and Jankovic, Mrdjan and Magner, Steve and Teel, Andrew R},
  journal={IEEE Transactions on Control Systems Technology},
  volume={14},
  number={3},
  pages={398--407},
  year={2006},
  publisher={IEEE}
}

@book{ref:khalil-2002,
      author        = "Khalil, Hassan K",
      title         = "{Nonlinear systems; 3rd ed.}",
      publisher     = "Prentice-Hall",
      address       = "Upper Saddle River, NJ",
      year          = "2002",
}

@INPROCEEDINGS{ref:teel-1998,
  author={Teel, A.R. and Peuteman, J. and Aeyels, D.},
  booktitle={Proceedings of the 37th IEEE Conference on Decision and Control (Cat. No.98CH36171)}, 
  title={Global asymptotic stability for the averaged implies semi-global practical asymptotic stability for the actual}, 
  year={1998},
  volume={2},
  number={},
  pages={1458-1463 vol.2},
  doi={10.1109/CDC.1998.758492}
}

@article{ref:jiang-1996,
	title = {A Lyapunov formulation of the nonlinear small-gain theorem for interconnected ISS systems},
	author = {Zhong-Ping Jiang and Iven M.Y. Mareels and Yuan Wang},
	journal = {Automatica},
	volume = {32},
	number = {8},
	pages = {1211-1215},
	year = {1996},
	issn = {0005-1098},
	doi = {https://doi.org/10.1016/0005-1098(96)00051-9},
	url = {https://www.sciencedirect.com/science/article/pii/0005109896000519}
}

@BOOK{ref:verhulst-1996,
  title     = "Nonlinear Differential Equations and Dynamical Systems",
  author    = "Verhulst, Ferdinand",
  publisher = "Springer",
  series    = "Universitext",
  edition   =  2,
  month     =  sep,
  year      =  1996,
  address   = "Berlin, Germany",
  language  = "en"
}

@Book{ref:fink-1975,
	author = {A.~M.~Fink},
	title = {Almost Periodic Differential Equations},
	publisher = {Springer Berlin, Heidelberg},
	year = {1975},
	series = {Lecture Notes in Mathematics},
	edition = {1},
}

@Book{ref:hahn-1967,
	author = {Wolfgang Hahn},
	title = {Stability of Motion},
	publisher = {Springer Berlin, Heidelberg},
	year = {1967},
	edition = {1},
}

\appendices

\section{Practical Stability Definitions}
	\label{app:practical-stability}

Let $x(t)$ be a solution to the differential equation
\begin{equation}
	\label{eq:practical-stability:differential-equation}
	\derivative{t}x = f^{\varepsilon}\left(t, x\right)
\end{equation}
where $f^{\varepsilon}:\realnumbers\times\realnumbers^n\to\realnumbers^n$ is parameterized by a vector of small parameters $\varepsilon = (\varepsilon_1,\ldots,\varepsilon_m) \in\strictlypositivereals^{m}$. Stability of trajectories in these systems are given in terms of practically stability as given in the following Definitions \ref{def:practically-stable}-\ref{def:sgpuas}. The reader should note that that the phase "there exists $\limitvalue{\varepsilon} > 0$ such that for all $\varepsilon\in (0,\limitvalue{\varepsilon})$" is a short hand for "there exists $\limitvalue{\varepsilon}_1 > 0$ such that for all $\varepsilon_1 \in (0,\limitvalue{\varepsilon}_1)$, there exists $\limitvalue{\varepsilon}_2$ such that for all $\varepsilon_2 \in (0,\limitvalue{\varepsilon}_2)$, ..." when dealing with multiple small parameters.

\begin{defn}[Practically Stable]
	\label{def:practically-stable}
	The origin is practically stable for \eqref{eq:practical-stability:differential-equation} if for every $c_2 > 0$, there exists $c_1 > 0$ and $\limitvalue{\varepsilon} > 0$ such that for all $\varepsilon\in (0,\limitvalue{\varepsilon})$,
	\begin{equation}
		\Vert x(t_0) \Vert < c_1 \implies \Vert x(t) \Vert < c_2\ \forall\ t\in[t_0,\infty)
	\end{equation}
\end{defn}

\begin{defn}[Practically Bounded]
	\label{def:practically-bounded}
	The origin is practically bounded of \eqref{eq:practical-stability:differential-equation} if for every $c_2 > 0$, there exists $c_1 > 0$ and $\limitvalue{\varepsilon} > 0$ such that for all $\varepsilon\in (0,\limitvalue{\varepsilon})$,
	\begin{equation}
		\Vert x(t_0) \Vert \leq c_1 \implies \Vert x(t) \Vert \leq c_2\ \forall\ t\in[t_0,\infty)
	\end{equation}
\end{defn}

\begin{defn}[$\delta$-Practically Uniformly Attractiveness]
	\label{def:practically-uniformly-attractive}
	The origin is $\delta$-practically uniformly attractive of \eqref{eq:practical-stability:differential-equation} if for $c_1 = \delta > 0$, for all $c_2 > 0$, there exists $T = T(c_1, c_2) > 0$ and $\limitvalue{\varepsilon} > 0$ such that for all $\varepsilon\in (0,\limitvalue{\varepsilon})$,
	\begin{equation}
		\Vert x(t_0) \Vert \leq c_1 \implies \Vert x(t) \Vert \leq c_2\ \forall\ t\in[t_0 + T,\infty)
	\end{equation}
	
	If a stricter bounds can be place on the trajectories
	\begin{equation}
		\Vert x(t) \Vert \leq c_1 M e^{-\lambda (t - t_0)} + c_2\ \forall\ t\in[t_0,\infty)
	\end{equation}
	where $M,\lambda>0$ are constants independent of $\delta$, then the origin is locally practically exponentially stable (PES).
	
	If $\delta$ can be made arbritrarily large, then the origin is said to be semiglobally practically uniformly attractive of \eqref{eq:practical-stability:differential-equation}.
\end{defn}

\begin{defn}[sGPUAS]
	\label{def:sgpuas}
	The origin is semiglobally practically uniformly attractively stable (sGPUAS) of \eqref{eq:practical-stability:differential-equation} if it is practically stable, practically bounded, and semiglobally practically uniformly attrictive of \eqref{eq:practical-stability:differential-equation}.
\end{defn}

If the system of \eqref{eq:practical-stability:differential-equation} is independent of the small parameters, then the notion of "practically" can be dropped and the stability definitions follow those presented in \cite[Ch~4]{ref:khalil-2002}.

\end{document}